\documentclass[11pt]{article}

\usepackage{amssymb}
\usepackage{amsmath}
\usepackage{amsfonts}
\usepackage{enumerate}
\usepackage{mathabx}
\usepackage{enumitem}

\usepackage{hyperref}
\usepackage[font=small,labelfont=bf]{caption}

\usepackage{tikz}
\usetikzlibrary{shapes.geometric, arrows}
\tikzstyle{startstop} = [rectangle, rounded corners, minimum width=3cm, minimum height=1cm,text centered, draw=black, fill=red!30]
\tikzstyle{arrow} = [thick,->,>=stealth]

\tikzstyle{arrow} = [thick,->,>=stealth]

\usepackage{parskip}

\usepackage{tikz}
\usepackage[makeroom]{cancel}
\usepackage{amsthm}
\usepackage{dsfont}

\usepackage{graphicx}

\usepackage{pgfplots}
\pgfplotsset{width=10cm,compat=1.17}

\newcommand{\bN}{ {\mathbb{N} } }

\newcommand{\bR}{ {\mathbb{R}} }
\newcommand{\bZ}{ {\mathbb{Z}} }

\newtheorem{theorem}{Theorem}
    \newtheorem{lemma}[theorem]{Lemma}
    
    \newtheorem{corollary}[theorem]{Corollary}

    \newtheorem{manualtheoreminner}{Theorem}
    \newenvironment{manualtheorem}[1]{%
        \renewcommand\themanualtheoreminner{#1}%
        \manualtheoreminner
    }{\endmanualtheoreminner}

\theoremstyle{definition} 
    \newtheorem*{definition}{Definition}
    \newtheorem*{remark}{Remark}

\usepackage{titling}
\usepackage{ytableau}

\pretitle{\begin{center}\Large\bfseries}
\posttitle{\par\end{center}\vskip 0.5em}
\preauthor{\begin{center}\Large\ttfamily}
\postauthor{\end{center}}

\title{Grid entropy in a "choose the best of $D$ samples" model }
\author{Alexandru Gatea \\ \quad \\ Department of Mathematics, University of Toronto \\ \href{mailto:alex.gatea@mail.utoronto.ca}{alex.gatea@mail.utoronto.ca} }

\begin{document}

\maketitle
\thispagestyle{empty}

\begin{abstract}
 Grid entropy is a deterministic quantity inherent to lattice models which captures the entropy of empirical measures along paths that converge weakly to a given target measure. In this paper, we study the limiting behaviour of empirical measures in a model consisting of repeatedly taking $D$ samples from a distribution and picking out one according to an omniscient "strategy." We show that the set of limit points of empirical measures is almost surely the same whether or not we restrict ourselves to strategies which make the choices independently of all past and future choices, and furthermore, that this set coincides with the set of measures with finite grid entropy. This set is convex and weakly compact; we characterize its extreme points as those given by a natural "greedy" deterministic strategy and we compute the grid entropy of said extreme points to be 0. This yields a description of the set of limit points of empirical measures as the closed convex hull of measures given by a density which is $D \cdot Beta(1,D)$ distributed. We also derive a simplified version of a grid entropy-based variational formula for Gibbs Free Energy for this model, and we present the dual formula for grid entropy.
\end{abstract}

\tableofcontents

\section{Introduction}

Random sampling is a fundamental area of study in statistics. Countless processes can be simulated as Monte Carlo experiments, with far-reaching applications in a range of fields from biology and economics to business. A rather simple but ubiquitous and surprisingly non-trivial  experiment is that of taking $D$ samples from a distribution and picking out one of these samples according to some omniscient "strategy." 

 This problem is dual to K-means clustering, a well-studied process where one groups a large set of samples into a fixed  $K$ number  of buckets in order to minimize an error function called the quadratic distortion (see \cite{liu2020convergence} for state-of-the-art).

We frame our model in the realm of percolation theory and ask what is the limiting behaviour of  empirical measures  along the event-dependent path of choices. Answering this question provides a gold mine of information about the model - normalized passage time along a path, after all, is just the identity function integrated against the path's empirical measure. 

Recent efforts have gone into better understanding the weak convergence of empirical measures. Some, such as \cite{riekert2021wasserstein}  and  \cite{lei2020convergence}, establish upper bounds on the \mbox{rate} of Wasserstein convergence of empirical measures to the Markov invariant measure/i.i.d. measure respectively they are sampled from. Others focus on describing the limits of empirical measures. 

In \cite{bates}, Bates shows that the set $\mathcal{R}$ of limit points of normalized empirical measures along paths in a generic lattice model is deterministic and derives an explicit variational formula for the limit shape of first passage time as the minimum value of a linear functional over this deterministic set. 

In \cite{gatea}, this author builds on Bates' paper by developing the notion of grid entropy, a deterministic quantity capturing not just whether a certain target measure is a limit point of empirical measures. Grid entropy has previously appeared in \cite{rassoul2014quenched}. \cite{gatea}
proposes a novel approach, realizing grid entropy both as a subadditive limit of entropies of paths with empirical measure within an $\epsilon$-Levy-Prokhorov distance of the target measure, and as the critical exponent of canonical order statistics associated with the Levy-Prokhorov metric, something a priori only known for i.i.d. Bernoulli$p)$ edge labels (see \cite{carmona2010directed}). Though using different approaches, both this author and Rassoul-Agha et al establish the same convex duality between grid entropy and Gibbs Free Energy. Furthermore, this author observes that Bates' set $\mathcal{R}$ almost surely coincides with the set of probability measures with finite grid entropy:
\[ \mathcal{R} = \{\nu \ \mbox{prob meas}: ||\nu|| > -\infty\} \ \mbox{a.s.} \]

One limitation of \cite{gatea} is that only empirical measures along deterministic or polymer paths are considered, rather than allowing for arbitrary mixtures of paths. Even in \cite{bates}, the focus is mainly on empirical measures along geodesics. In this paper, we seek to rectify this by studying empirical measures along  \emph{random} paths (picked according to some probabilistic "strategy," which may or may not yield geodesics or polymer paths). We will show that $\mathcal{R}$ is still the set of limit points, whether or not we assume the strategy is omniscient.

Another deficiency is that grid entropy and the set $\mathcal{R}$, like limit shape and other quantities in this area, are not known to be explicitly computable in most cases. One noteworthy exception is the paper \cite{martinAllan} by Martin, in which an explicit formula for the weak limits of empirical measures along geodesics is derived in the solvable Exponential LPP on $\bZ^2$ model. Also, in \cite{bates} it is established that replacing the lattice $\bZ^D$ by the infinite complete $D$-ary tree $\mathcal{T}_D$, the set $\mathcal{R}$ can be precisely described as a specific sublevel of relative entropy (in terms of $D$).

The $D$-ary tree is in a way the dual model to the one  we explore in this paper, and as one might expect, we will be able to give a description of $\mathcal{R}$ in our setting too. We will precisely characterize the extreme points of $\mathcal{R}$ as those measures whose density is $D \cdot Beta(1,D)$ distributed; $\mathcal{R}$ is then the closed convex hull of these measures.

In addition, we compute the grid entropy of the extreme points of $\mathcal{R}$ to be 0, and we give a simplification of a general formula for grid entropy from \cite{gatea} for this model.

\section{Definitions and  Results}
Let us be more precise about our setup and goals. We consider vertices on $\bZ_{\geq 0}$, $D$ parallel edges $(e_i^1,\ldots, e_i^D)$ for every $i \geq 0$, and an i.i.d. array of edge labels $U_i^j \sim $ Unif[0,1]. We denote by $\Lambda$ the Lebesgue measure on $[0,1]$.

It is convenient to work on the compact space $\mathcal{M}_1$ of probability measures on [0,1], and as we will explain in Section \ref{coupling} we do not lose any generality by restricting to this setting.


\begin{center}
\begin{tikzpicture}
\node (p1) at (0pt,0pt) {};
\node (p2) at (50pt,0pt) {};
\node (p3) at (100pt,0pt) {};
\filldraw (50pt,0pt)circle(2pt) (0pt,0pt)circle(2pt)(100pt,0pt)circle(2pt) 
(110pt,0pt)circle(2pt)(120pt,0pt)circle(2pt)(130pt,0pt)circle(2pt);

\draw (25pt,0pt) node {$e_0^2$};
\draw (25pt,16pt) node [ above ] {$e_0^1$};
\draw (25pt,-16pt) node [ below ] {$e_0^3$};
\draw (75pt,0pt) node {$e_1^2$};
\draw (75pt,16pt) node [ above ] {$e_1^1$};
\draw (75pt,-16pt) node [ below ] {$e_1^3$};
\draw (0pt,0pt) node [ below right] {$0$};
\draw (50pt,0pt) node [ below right] {$1$};
\draw (100pt,0pt) node [ below right] {$2$};

\draw [->] (p1) to [out=90,in=90] (p2);
\draw [->] (p1) to [out=-90,in=-90] (p2);
\draw [->] (p2) to [out=90,in=90] (p3);
\draw [->] (p2) to [out=-90,in=-90] (p3);

  \begin{scope}[every path/.style={->}]
       \draw (p1) -- (p2); 
       \draw (p2) -- (p3);
    \end{scope}  

\end{tikzpicture}
\end{center}


Next, we wish to formalize the notion of  strategies. These will be probability measures on the product space coupling the environment $(U_i^j)_{i \geq 0, 1 \leq j \leq D}$ with  the infinite target sequence of indices of the $U_i$ corresponding to the choices $(X_i)_{i \geq 0}$.

\begin{definition}
A strategy  is a probability measure $\chi$ on the product space\\ $([0,1]^D)^{\bZ_{\geq 0}} \times \{1,\ldots,D\}^{\bZ_{\geq 0}}$ s.t. the marginal distribution of  the first coordinate (the environment $(U_i^j) \in ([0,1]^D)^{\bZ_{\geq 0}}$) is a sequence of i.i.d. Unif[0,1]. 

We denote by $(J_i)_{i \geq 0}$  the second coordinate (a random sequence of indices) and define the random vector of choices
\[ (X_0(\chi), X_1(\chi),\ldots) := (U_0^{J_0}, U_1^{J_1},\ldots) \]
Denote by $\frac1n \mu_{0 \rightarrow n}(\chi)$ the empirical measures of this vector:
\[ \frac1n \mu_{0 \rightarrow n}(\chi) = \frac1n \sum \limits_{i=0}^{n-1} \delta_{X_i(\chi)} \in \mathcal{M}_1\]
Also, let $\sigma_{\chi}^i$ be the law of $X_i(\chi)$.

\end{definition}

\begin{remark}
If $\chi$ conditioned on the environment is a delta mass for a.a. environments $(U_i^j)$ with respect to the product measure  $(\Lambda^{\times D})_{\infty}$, it means that the strategy picks exactly one sequence $(x_i)_{i \geq 0}$ for each set of observed labels. In other words, the strategy is  deterministic. 
\end{remark}

\begin{definition}
We also define (strategy-free) empirical measures along a fixed path $\pi: 0 \rightarrow n$ consisting of edges $(e_0^{j_0},\ldots, e_{n-1}^{j_{n-1}})$ by
\[  \frac1n \mu_{\pi} = \frac1n \sum_{i=0}^{n-1} \delta_{U_i^{j_i}}  \]

\end{definition}

An important type of strategy is one which chooses each $X_k$ independent of $(U_i^j)_{i \neq k, 1 \leq j \leq D}$, the observed values from all but the $k$th trial. These are $\chi$ which are  infinite products of measures arising from "single-edge strategies," i.e. micro-strategies operating at the individual trial level. We call such $\chi$ "independent strategies," as they give rise to independent $X_i$.

\begin{definition}
A single-edge strategy    is a probability measure $\psi$ on the product space\\ $([0,1]^D) \times \{1,\ldots,D\}$ s.t. the marginal distribution of  the first coordinate $(U^j)_{1 \leq j \leq D}$ is a sequence of $D$ i.i.d. Unif[0,1].
We denote by $J$  the second coordinate (a random index) and define the random choice
\[ X(\psi):= U^J \]
Also, let $\sigma_{\psi}$ be the law of $X(\psi)$.
\end{definition}
\begin{remark}
If $\psi$ conditioned on $(U^j)_{1 \leq j \leq D}$ is a delta mass for a.a. $(u^j)_{1 \leq j \leq D}$ with respect to the product measure  $\Lambda^{\times D}$, it means that the single-edge strategy picks exactly one sequence $x$ for each set of $D$ observed labels, so $\psi$ is  deterministic.
\end{remark}

 A single-edge strategy $\psi$ is completely determined by the $\psi$-conditional probabilities \\ $p_k: [0,1]^D \rightarrow [0,1]$, $1 \leq k \leq D$ defined as
 \[ p_k(u_1,\ldots, u_D) := P_{\psi}[J=k \mid (U^j)_{1 \leq j \leq D} = (u^j)_{1 \leq j \leq D} ]  \]
We expand on this in more detail in Section \ref{vectorFcns}. The key takeaway is that we may interchangeably refer to both $\psi$ and $\vec{p}$ as a single-edge strategy, and therefore we define $X(\vec{p}):=X(\psi), \sigma_{\vec{p}} :=\sigma_{\psi}$.

\bigskip

In this paper we are interested in the weak limit points of $\frac1n \mu_{0 \rightarrow n}(\chi)$. As it turns out, this set of limit points almost surely coincides with the set of limit points of $\frac1n \mu_{0 \rightarrow n}(\chi)$ over independent strategies $\chi$ only, with the distributions $\sigma_{\vec{p}}$  over single-edge strategies $\vec{p}$, as well as with the set $\mathcal{R}$ of probability measures with finite grid entropy. The following theorem is the main objective of Section \ref{indepReduction} .

 \begin{manualtheorem}{A}\label{part2_A}
A.s. we have
\begin{align*}
    \mathcal{R} &= \bigg\{\mbox{limit pts of} \ \frac1n \mu_{0\rightarrow n}(\chi):  \mbox{strategies} \ \chi \bigg\} \\
    &=\bigg\{\mbox{limit pts of} \ \frac1n \mu_{0\rightarrow n}(\chi):  \mbox{independent strategies} \ \chi \bigg\}\\
    &= \bigg\{\sigma_{\vec{p}}:  \mbox{single-edge strategies} \ \vec{p} \bigg\}
\end{align*}
\end{manualtheorem}

This reduces the problem to working with single-edge strategies $\vec{p}$, which are simpler to handle than generic strategies.  Moreover, every $\sigma_{\vec{p}}$ can be achieved by a "consistent" single-edge strategy $\vec{p}$ with the properties that $p_1(\vec{u})$ is invariant under 1-fixing permutations in $S_D$ and $p_k(\vec{u}) = p_1(u_{D-k+2}, \ldots, u_1,\ldots, u_{D-k+1}) \ \forall k$.



Now, the sets from Theorem \ref{part2_A} are convex and weakly compact. In Section \ref{extreme} we fully characterize their extreme points.

 \begin{manualtheorem}{B}\label{B}
Let $\vec{p}$ be a single-edge strategy.  The following are equivalent:
\begin{enumerate}[label=(\roman*)]
\item $\sigma_{\vec{p}}$ is an extreme point
\item Any consistent single-edge strategy achieving $\sigma_{\vec{p}}$  must be deterministic.
\item $\sigma_{\vec{p}}$ has a density $f_{\vec{p}}$ which is not constant on sets of positive $\Lambda$-measure and $\sigma_{\vec{p}}$ is given by the following single-edge strategy $\vec{q}$:
\[ q_k(u_1,\ldots, u_D) = \textbf{1}_{\{f_{\vec{p}}(u_k) \geq f_{\vec{p}}(u_i) \ \forall i\}} \ \mbox{for $\Lambda^{\times D}$-a.a.} \ (u_1,\ldots, u_D) \]
In other words, $\sigma_{\vec{p}}$ is achieved by  the deterministic greedy single-edge strategy "choose whichever label yields a higher value when evaluating the density $f_{\vec{p}}$."
\item $\sigma_{\vec{p}}$ has a density $f_{\vec{p}}$ s.t.  $\frac1D f_{\vec{p}}(Unif[0,1]) \sim \mbox{Beta}(1,D)$.
\end{enumerate}
\end{manualtheorem}

As an immediate corollary, $\mathcal{R}$ is the closed convex hull of these measures.

\begin{remark}
It is important to stress that Theorem \ref{B} holds even in a more general setting where the i.i.d. labels $U_i^j$ follow some finite mean distribution on $\mathcal{R}$ that is not necessarily Unif[0,1]. The only caveat is that the value distribution of the densities $f_{\vec{p}}$ of extreme points might not have as explicit a form as $D \cdot Beta(1,D$); however, the value distribution of the densities $f_{\vec{p}}$ is still identical to the value distribution of $f_{MAX}$, whatever that may be.
\end{remark}

\begin{remark}
In fact, Theorem \ref{B} holds in a discrete setting as well. If the i.i.d. labels $U^j$ are Unif$\{1,\ldots, K\}$ then the convex set of consistent single-edge strategies forms a permutohedron of order $K$.
\end{remark}
\begin{center}
\captionsetup{type=figure}
\includegraphics[scale=0.4]{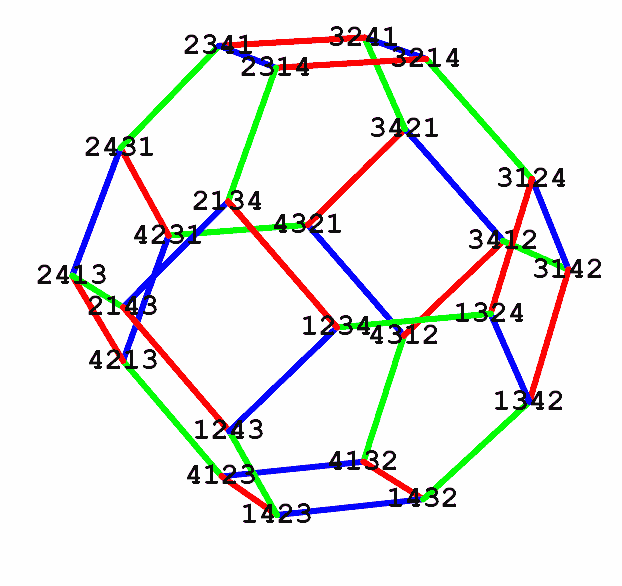}
\captionof{figure}{Permutohedron of order 4, \cite{holroyd}}
\end{center}
The following is the discrete analogue of Theorem \ref{B}. 
\begin{manualtheorem}{B'}\label{B'}
Let $\vec{p}$ be a single-edge strategy.  The following are equivalent:
\begin{enumerate}[label=(\roman*)]
\item $\sigma_{\vec{p}}$ is an extreme point
\item Any consistent single-edge strategy achieving $\sigma_{\vec{p}}$  must be deterministic.
\item $\sigma_{\vec{p}}$  is given by the single-edge strategy "choose whichever label is maximal with respect to the ordering $\alpha(1) < \cdots < \alpha(K)$"  for some permutation $\alpha \in S_K$.
\item $\sigma_{\vec{p}}$ has a probability mass function whose value distribution is 
\[\bigg\{\frac1{K^D}, \frac{2^D-1}{K^D}, \frac{3^D-2^D}{K^D}, \ldots, \frac{K^D-(K-1)^D}{K^D} \bigg\}\]
\end{enumerate}
\end{manualtheorem}
The ordering mentioned in (iii) gives the natural bijection between the extreme consistent single-edge strategies and the extreme points of the permutohedron. For example, the extreme point $1342$ corresponds to the single-edge strategy of choosing the maximal $u^i$ according to the ordering $1<3<4<2$,  hence the bijection maps
\[ 1342 \mapsto \sigma_{1342} = \frac1{4^D} \delta_1 + \frac{2^D-1}{4^D} \delta_3 + \frac{3^D-2^D}{4^D} \delta_4 + \frac{4^D-3^D}{4^D} \delta_2 \]
This bijection extends to a bijection between the entire permutohedron and $\mathcal{R}$ by taking convex combinations. We walk through an explicit example of this phenomenon in Section \ref{discrete}.

\bigskip

Generally speaking there is no known way of computing grid entropy, however, it amazingly can be computed to be 0 for these extreme points. This effectively means that a.s. the number of paths $0 \rightarrow n$ with empirical measures weakly converging to any one of these extreme points is $e^{o(n)}$.

 \begin{manualtheorem}{C}
Let $\sigma_{\vec{p}}$ be an extreme point. Then
\[ || \sigma_{\vec{p}} || = 0 \]
\end{manualtheorem}
Whether \emph{every} measure in $\mathcal{R}$ with grid entropy 0 is an extreme point remains an open question.

Section \ref{grid} focuses on this result, as well as simplified formulas for the Gibbs Free Energy and grid entropy in this model.

 \begin{manualtheorem}{D}
Suppose $\tau: [0,1] \rightarrow \bR$ is a bounded measurable function. Then Gibbs Free Energy with respect to $\tau$ is given by 
\[ G(\tau) = E \bigg[\log \sum_{j=1}^D e^{\tau(U^j)} \bigg] \ \mbox{a.s.}\]
where $U^j$ are i.i.d. Unif[0,1]. For all probability measures $\nu$, grid entropy is given by
\[ -||\nu|| = \sup_{\tau} \bigg[ \langle \tau, \nu \rangle - G(\tau) \bigg] = \sup_{\tau} \bigg[ \langle \tau, \nu \rangle - E \bigg[\log \sum_{j=1}^D e^{\tau(U^j)} \bigg] \bigg] \]
where the supremum is over bounded measurable functions $\tau:[0,1] \rightarrow \bR$ and where $\langle \tau, \nu \rangle$ denotes the integral $ \int_0^1 \tau(u) d\nu$.
\end{manualtheorem}
\begin{remark}
This theorem actually holds for measurable $\tau$ s.t. $E[e^{\beta \tau(U)}] < \infty$ for $U \sim Unif[0,1]$ and can easily be extended to all measurable $\tau$ by truncating $\tau$ in $\langle \tau, \nu \rangle$ at some $C>0$ and taking a supremum over $C$ of the variational formula. We leave the details to the reader and focus on the bounded $\tau$ case.
\end{remark}

But first we delve deeper into the setup and known relevant results.

\section{Preliminaries}

\subsection{More on Strategies}
We list some miscellaneous observations about strategies.




It is trivial to see that the sets of strategies $\chi$, of single-edge strategies $\psi$, of distributions $\sigma_{\chi}^i$ for strategies $\chi$ and $i \geq 0$, and of $\sigma_{\psi}$ for single-edge strategies $\psi$ are each closed under convex combinations. The extreme points of the sets of strategies/single-edge strategies are clearly the sets of deterministic strategies/single-edge strategies respectively.

Moreover, if $\chi$ is an independent strategy corresponding to a sequence $(\psi_i)$ of single-edge strategies then $\sigma_{\chi}^i = \sigma_{\psi_i} \ \forall i \geq 0$. In particular, this implies
\begin{equation} \label{EQ2}
    \{\sigma_{\chi}^i: \mbox{independent strategies} \ \chi\} = \{\sigma_{\psi}: \mbox{single-edge strategies} \ \psi\} \ \forall i \geq 0
\end{equation} 

In fact, observe that for any strategy $\chi$ and $i \geq 0$ if we define the strategy $\chi'$ to be the measure determined on product sets by
\[ \chi'(A \times B) := \chi( (0 \ i) A \times (0 \ i) B) \ \forall \  \mbox{measurable} \ A \subseteq ([0,1]^D)^{\bZ_{\geq 0}}, B \subseteq (\{1,\ldots,D\})^{\bZ_{\geq 0}}\]
where $(0 \ i) C$ swaps the 0th and $i$th coordinates of sequences in $C$, then this is easily checked to be a strategy with
\[ \sigma_{\chi}^i = \sigma_{\chi'}^0 \]
Therefore
\begin{equation} \label{EQswap}
    \{\sigma_{\chi}^i: \mbox{strategies} \ \chi\} =  \{\sigma_{\chi}^0: \mbox{strategies} \ \chi\} \ \forall i \geq 0
\end{equation} 

\subsection{Coupling the Edge Weights with I.I.D. Uniforms}\label{coupling}
We briefly explain why we can work in the setting of edge labels $U_i^j \sim Unif[0,1]$ without losing generality.

Suppose instead we start with i.i.d. edge weights $(\tau_i^j)_{i \geq 0, 1 \leq j \leq D} \sim \theta$ for some distribution $\theta$ on $\bR$ with finite mean. We can then couple the environment to uniform random variables as is done in   \cite[Sect.~2.1]{bates}. More specifically, we let our the $\tau_i^j$ be given by
\[ \tau_i^j = \tau(U_i^j) \]
for a measurable function $\tau: [0,1] \rightarrow \bR$ and i.i.d. $U_i^j \sim Unif[0,1]$. If we let $F_{\theta}$ be the cumulative distribution function of $\theta$ then the quantile function 
\[ F_{\theta}^{-}(t) := \inf \{t \in \bR: F_{\theta}(t) \geq x\} \]
is an example of such a $\tau$;  however, our results hold for any such $\tau$ chosen so we let $\tau$ be arbitrary.

We can move from the simplified $(\Lambda, [0,1])$ setting we wish to work in to the initial $(\theta, \bR)$ setting by applying the $\tau$-pushforward.

Now, as mentioned in Theorem \ref{part2_A}, we show that, in the $(\Lambda, [0,1])$ setting, the set of limit points of the empirical measures $\frac1n \mu_{0\rightarrow n}(\chi)$ a.s. coincides with $\mathcal{R}$, which is shown in \ref{gridEntropyPart1} to a.s. coincide with the set of limit points of (strategy-free) empirical measures $\frac1n \mu_{\pi_n}$ along paths $\pi_n:0 \rightarrow n$.

We do not lose generality by performing this coupling because a.s., $\frac1{n_k} \mu_{\pi_{n_k}} \Rightarrow \nu$ implies $\tau_{\ast}(\frac1{n_k} \mu_{\pi_{n_k}}) \Rightarrow \tau_{\ast}(\nu)$.
This is proved in  \cite[Lemma 6.15]{bates}. In fact, if we take $\tau $ to be the quantile function and we assume $\theta$ has continuous cdf $F_{\theta}$ then this implication becomes an if and only if, as seen in \cite[Lemma 5]{gatea}.

\subsection{Grid Entropy}

\begin{definition}The Levy-Prokhorov metric  on the space $\mathcal{M}_+$ of finite non-negative Borel measures on $[0,1]$ is defined by 
\[\rho(\mu, \nu) = \inf \{\epsilon > 0: \mu(A) \leq \nu(A^{\epsilon}) + \epsilon \ \mbox{and} \ \nu(A) \leq \mu(A^{\epsilon}) + \epsilon \ \forall A \in \mathcal{B}([0,1])\}\]
\end{definition}

It is standard that the Levy-Prokhorov metric $\rho$  metrizes weak convergence. Some elementary properties include that the Levy-Prokhorov metric is upper bounded by total variation and it satisfies a certain subadditivity:
\[ \rho(\mu_1 + \mu_2, \nu_1 + \nu_2) \leq \rho(\mu_1, \nu_1) + \rho(\mu_2, \nu_2) \]

For $t \geq 0$ let $\mathcal{M}_t$ denote the space of non-negative Borel measures on $[0,1]$ with total mass $t$.

In a concurrent paper \cite{gatea}, this author studies not just the convergence of empirical measures in a lattice model such as ours but the \emph{entropy} of empirical measures converging to a certain weak limit.

Three quite different but equivalent definitions of grid entropy are given in said paper, and they are linked to the original description of this entropy provided in a 2014 paper \cite{rassoul2014quenched} by Rassoul-Agha and Sepp{\"a}l{\"a}inen. Here we work with the two most relevant to our needs, coming from \cite{gatea}.  Note that in the setting of \cite{gatea}, both direction-fixed and direction-free grid entropy are considered, but in the model we focus on in this paper there is but one unit direction so we will use the direction-free versions of the results from \cite{gatea}. 

We go into a brief description of  this work, as fits our needs in this paper.

Fix any $t \geq 0$ and a target measure $\nu$. We consider the order statistics of the Levy-Prokhorov distance between $\nu$ and the   empirical measures $\frac1n \mu_{\pi}$ varying over all $D^{\lfloor nt\rfloor}$ possible origin-anchored, length $\lfloor nt \rfloor$ paths $\pi$. That is, for every $n \in \bN$ we let
\[ \min_{\pi: 0\rightarrow \lfloor nt \rfloor}^1 \rho\bigg(\frac1n \mu_{\pi}, \nu \bigg) \leq \min_{\pi: 0\rightarrow \lfloor nt \rfloor}^2   \rho\bigg(\frac1n \mu_{\pi}, \nu \bigg) \leq \ldots \leq  \min_{\pi: 0\rightarrow \lfloor nt \rfloor}^{D^{\lfloor nt \rfloor}}  \rho\bigg(\frac1n \mu_{\pi}, \nu \bigg) \]
denote the order statistics value of $ \rho(\frac1n \mu_{\pi}, \nu)$. It is convenient to define
\[ \min_{\pi: 0\rightarrow \lfloor nt \rfloor}^j \rho\bigg(\frac1n \mu_{\pi}, \nu \bigg) := +\infty \ \mbox{for} \ j > D^{\lfloor nt \rfloor}\]

 These order statistics and the paths corresponding to them are of course event-dependent. However,  the following theorem from \cite{gatea}
 shows there is a deterministic critical exponent where the empirical measures along the paths corresponding to these order statistics change from converging a.s. to $\nu$ to  a.s. diverging away from $\nu$. This critical exponent turns out to be precisely the negative convex conjugate of Gibbs Free Energy. We state the full theorem as presented in \cite{gatea}, noting that part (i) has been previously shown in \cite{carmona2010directed} to hold in the case of Bernoulli edge labels, and that parts (ii) and (iv) follow immediately from the earlier work of Rassoul-Agha and Sepp{\"a}l{\"a}inen   \cite{rassoul2014quenched}. 
 
\begin{theorem}\label{gridEntropyPart1}
\begin{enumerate}[label=(\roman*)]
\item[]
\item For any target  measure $\nu \in \mathcal{M}$, its grid entropy is defined to be the deterministic quantity
\begin{align*}
 || \nu|| &:= \sup \bigg\{\alpha \geq 0: \lim_{n \rightarrow \infty} \min_{\pi: 0 \rightarrow \lfloor nt \rfloor}^{\lfloor e^{\alpha n} \rfloor}  \rho\bigg(\frac1{n} \mu_{\pi}, \nu \bigg) = 0  \ \mbox{a.s.} \bigg\} \\
 &\in \{-\infty\} \cup [0, t\log D]
 \end{align*}
 where the value is $-\infty$ if the set of $\alpha$'s is empty. Then  grid entropy is the critical exponent where the $\min \limits_{\pi: 0\rightarrow \lfloor nt \rfloor}^{\lfloor e^{\alpha n} \rfloor} \rho(\frac1n \mu_{\pi}, \nu )$ change from converging  to 0 to a.s. having a $\liminf \limits_{n \rightarrow \infty} >0$.\\
For example, the grid entropy of the original distribution, $\Lambda$, is maximal:\\ $||\Lambda|| = t \log D$.
\item Grid entropy is positive-homogeneous, satisfies the reverse-triangle inequality
\[ ||\nu|| + ||\xi|| \leq || \nu + \xi||\]
and is concave and upper-semicontinuous in $\nu$.
\item Consider the deterministic, weakly closed set 
\[ \mathcal{R}^t:= \{\nu \in \mathcal{M}: ||\nu|| \geq 0 \}\]
Then $\mathcal{R}^t \subseteq \mathcal{M}_t$,  $\nu \ll \Lambda \ \forall \nu \in \mathcal{R}^t$ and 
\[ \mathcal{R}^t = \bigg\{\mbox{limit pts of} \ \frac1n \mu_{\pi} \ \mbox{for} \ \pi: 0\rightarrow \lfloor nt \rfloor \bigg\} \ \mbox{a.s.} \]
\item Grid entropy is the negative convex conjugate of $\beta$-Gibbs Free Energy. More concretely, for $\beta > 0$,
\begin{equation}\label{dualFormula} -||\nu|| = \sup_{\tau} \bigg[ \beta \langle \tau, \nu \rangle - G_t^{\beta}(\tau) \bigg] \ \forall \nu \in \mathcal{M}_t
\end{equation}
where the supremum is taken over bounded measurable $\tau: [0,1] \rightarrow \bR$, where \\ $\langle \tau, \nu \rangle = \int_0^1 \tau(u) d\nu$, where $G_t^{\beta}(\tau)$ is the length $t$ $\beta$-Gibbs Free Energy
\[ G_t^{\beta}(\tau) := \lim_{n \rightarrow \infty} \frac1n \log \sum_{\pi: 0 \rightarrow \lfloor nt \rfloor} e^{\beta T(\pi)} \]
and where $T(\pi) = \sum \limits_{e\in \pi} \tau(U_e) = \langle \tau, \mu_{\pi} \rangle$ is the passage time along $\pi$.
\item Any $\nu \in \mathcal{R}^t$ satisfies the following upper bound on the sum of the grid entropy and the relative entropy with respect to Lebesgue measure on [0,1]:
$$t D_{KL} (\nu||\Lambda) + ||\nu|| \leq  t \log D$$
 where $D_{KL}$ denotes relative entropy (or Kullback-Leibler divergence).
\end{enumerate}
  \end{theorem}
  
  \begin{remark}
  An immediate but highly non-trivial consequence of (i) is the existence of random length $\lfloor nt \rfloor$ paths whose empirical measures converge weakly to a given target $\nu \in \mathcal{R}^t$ a.s.. Indeed, the event-dependent paths corresponding to $\min \limits_{\pi: 0 \rightarrow \lfloor nt \rfloor}^1 \rho(\frac1n \mu_{\pi}, \nu)$ do the job.
  \end{remark}

\begin{remark}
(i), seen as the statement that the negative convex conjugate of Gibbs Free Energy is the critical exponent of the order statistics $\min_{\pi: 0\rightarrow \lfloor nt \rfloor}^j \rho(\frac1n \mu_{\pi}, \nu)$, has been previously proved in the case of Bernoulli($p$) edge labels in \cite{carmona2010directed} [Corollary 2]. 
\end{remark}

By positive-homogeneity it suffices to work with the case $t=1$. For the sake of notation we drop the 1 in $\mathcal{R}^1, G_1^{\beta}(\tau)$ for the rest of the paper.

Once we determine the extreme points of $\mathcal{R}$, we compute their grid entropies to be 0. We also give a limit-free formula for the $\beta$-Gibbs Free Energy $G^{\beta}(\tau)$ in our model, which renders the convex duality formula \eqref{dualFormula} more practical. 

\subsection{Measure-Preserving Bijections in \texorpdfstring{$\bR^n$}{Rn}}
We will encounter measure-preserving bijections in $\bR$, so we briefly outline the required notions in this section.

\begin{definition}
Let $\mathcal{B}(\bR^n)$ denote the Borel $\sigma$-algebra of $\bR^n$ and let $\mu$ be a Borel probability measure on $\bR^n$. Let $A, A' \in \mathcal{B}(\bR^n)$ with $\mu(A') = \mu(A) > 0$. A bijection $\phi: A \rightarrow A'$ is said to be $\mu$-measure-preserving if $\mu(B) = \mu(\phi(B)) $ for all $B \in \mathcal{B}(\bR^n), B \subseteq A$.
\end{definition}

Nishiura proves in \cite[Theorem 4]{nishiura} the existence of a $\mu$-measure-preserving bijection between Borel sets of equal, positive but not full $\mu$-measure, given that $\mu$ is nonatomic. We will only need the following simplified version of this theorem.

\begin{theorem}\label{thmNishiura}
Suppose $\theta$ is a Borel probability measure on $[0,1]$ that is absolutely continuous with respect to Lebesgue measure $\Lambda$ on $[0,1]$. Let $A, A' \in \mathcal{B}([0,1])$ s.t. $\theta(A) = \theta(A') \in (0,1)$. Then there exists a $\theta$-measure-preserving bijection $\phi: A \rightarrow A'$.
\end{theorem}

Preliminaries finished, our next goal is to show that we may restrict ourselves to working with single-edge strategies without losing generality.

\section{Reducing the Problem}\label{indepReduction}
Over the course of this section, we prove the following theorem, which allows us to reduce the problem of describing the set of limit points of the empirical measures $\frac1n \mu_{0\rightarrow n}(\chi)$ to characterizing the set of achievable distributions $\sigma_{\psi}$ for single-edge strategies $\psi$.

\begin{theorem}\label{limitPts}
A.s., the following sets are equal:
\begin{enumerate}[label=(\roman*)]
\item $\{\mbox{limit pts of} \ \frac1n \mu_{0\rightarrow n}(\chi):  \mbox{ strategies} \ \chi\}$
\item $\mathcal{R} := \{\nu \in \mathcal{M}_1: ||\nu|| \geq 0\}$
\item $\{\sigma_{\chi}^0: \mbox{ strategies} \ \chi\}$
\item $\{\sigma_{\psi}: \mbox{single-edge strategies} \ \psi\}$
\item $\{\mbox{limit pts of} \ \frac1n \mu_{0\rightarrow n}(\chi):  \mbox{independent strategies} \ \chi\}$
\item $\{\mbox{limit pts of} \ \frac1n \mu_{0\rightarrow n}(\chi):  \mbox{i.i.d. strategies} \ \chi\}$
\end{enumerate}
\end{theorem}

\begin{remark}
A priori it not clear that  (iii)-(iv) are weakly closed. We show this as part of our proof.
\end{remark}

We argue via a chain of inclusions. We begin with the most trivial of these.

\begin{lemma}\label{1st}
A.s. we have
\[ \bigg\{\mbox{limit pts of} \ \frac1n \mu_{0\rightarrow n}(\chi):  \mbox{ strategies} \ \chi \bigg\} \subseteq \mathcal{R}  \]

\end{lemma}

\begin{proof}
Fix a strategy $\chi$. Suppose $\frac1{n_k} \mu_{0 \rightarrow n_k}(\chi) \Rightarrow \nu$ for some observed edge labels \\ $(u_i^j)_{i \geq 0, 1 \leq j \leq D}$. Then letting $\pi_{n}: 0 \rightarrow n$ be the paths corresponding to the indices \\ $(J_0(\chi), \ldots, J_{n-1}(\chi))$ conditioned on these same observed edge labels, we get\\  $ \frac1{n_k} \mu_{\pi_{n_k}}=\frac1{n_k} \mu_{0 \rightarrow n_k}(\chi) \Rightarrow \nu$.  Thus $\nu$ is a limit point of the $\frac1n \mu_{\pi}$. 

The desired inclusion follows from
\[ \mathcal{R} = \bigg\{\mbox{limit pts of} \ \frac1n \mu_{\pi}:  \mbox{ paths} \ \pi:0 \rightarrow n \bigg\} \ \mbox{a.s.} \ \mbox{by Theorem \ref{gridEntropyPart1}(iii)}\]
\end{proof}

\subsection{Expected Value of Empirical Measures}\label{sect31}

For any  strategy $\chi$,
\[E[\delta_{X_i(\chi)}] =  \int_{[0,1]} \delta_x d\sigma_{\chi}^i(x) = \sigma_{\chi}^i \ \mbox{and} \ E \bigg[\frac1n \mu_{0\rightarrow n}(\chi)\bigg] = \frac1n \sum_{i=0}^{n-1} \sigma^i_{\chi} \]
Recalling \eqref{EQswap} and the fact that $\{\sigma^0_{\chi}: \mbox{strategies} \ \chi\}$ is closed under convex combinations, it follows that every $E [\frac1n \mu_{0\rightarrow n}(\chi)]$ is contained in $\{\sigma^0_{\chi}: \mbox{strategies} \ \chi\}$. Therefore
\begin{align*}
\bigg\{E \bigg[\frac1n \mu_{0\rightarrow n}(\chi)\bigg]:  \mbox{ strategies} \ \chi \bigg\} &= \{E[\delta_{X_0(\chi)}]:  \mbox{ strategies} \ \chi\} \\
&= \{\sigma_{\chi}^0:  \mbox{ strategies} \ \chi \}
\end{align*}
Also it is clear that
\[ \{E[\delta_{X(\psi)}]:  \mbox{single-edge strategies} \ \psi\} = \{\sigma_{\psi}:  \mbox{ single-edge strategies} \ \psi \}\]

A simple Tonelli argument establishes that the closures of the  two sets above coincide.

\begin{lemma} \label{equivalenceIndep}
\[ \mbox{cl}\{\sigma_{\psi}: \mbox{ single-edge strategies} \ \psi \} = \mbox{cl}\{\sigma^0_{\chi}: \mbox{ strategies} \ \chi \}  \]
\end{lemma}

\begin{proof} By \eqref{EQ2} $\{\sigma_{\psi}\} = \{\sigma^0_{\chi}: \mbox{indep strategies} \ \chi\}$ so it suffices to show
\[ \mbox{cl}\{\sigma_{\psi}: \mbox{ single-edge strategies} \ \psi \} \supseteq \{\sigma^0_{\chi}: \mbox{ strategies} \ \chi \} \]
Given a strategy $\chi$ and observed "future" edge labels $(u_i^j)_{i \geq 1, 1 \leq j \leq D} \in [0,1]^{\bZ_{\geq 1}}$, define $\psi_{(u_i^j)_{i \geq 1, 1 \leq j \leq D}}$ to be $\chi$  conditioned on the rest of the environment $(u_i^j)_{i \geq 1, 1 \leq j \leq D}$. 

The $\psi_{(u_i^j)_{i \geq 1, 1 \leq j \leq D}}$ are easily seen to be single-edge strategies. Since integrating over the entire environment $(U_i^j)_{i \geq 0, 1 \leq j \leq D}$ is equivalent to integrating over $(U_0^j)_{1 \leq j \leq D}$ first and then over $(U_i^j)_{i \geq 1, 1 \leq j \leq D}$, we have 
\[ \sigma_{\chi}^0 = \int_{([0,1]^D)^{\bZ_{\geq 1}}} \sigma_{\psi_{(u_i^j)_{i \geq 1, 1 \leq j \leq D}}} d\Lambda^{\infty}((u_i^j)_{i \geq 1, 1 \leq j \leq D})  \]
By  convexity and weak closure it follows that 
\[\sigma_{\chi}^0 \in   \mbox{cl}\{\sigma_{\psi}: \mbox{ single-edge strategies} \ \psi \}  \]
\end{proof} 

Next, we show that this closure of $\{\sigma_{\chi}^0: \mbox{strategies} \ \chi\}$ contains all probability measures with finite grid entropy. 

\begin{lemma} \label{3rd}
We have
\[\mathcal{R}:= \{\nu \in \mathcal{M}_1: ||\nu|| \geq 0\}  \subseteq \mbox{cl}\{E[\delta_{X_0(\chi)}]: \mbox{ strategies} \ \chi \} \]
as deterministic sets.
\end{lemma}

\begin{proof}
Suppose this is not the case, say  $\exists \nu \in \mathcal{R} \cap (\mbox{cl}\{E[\delta_{X_0(\chi)}]: \mbox{ strategies} \ \chi \})^C$. Thus there exists $\epsilon > 0$ s.t. $B_{\epsilon}(\nu) \cap cl\{E[\delta_{X_0(\chi)}]: \chi\} = \emptyset$.

Let $\pi_n: 0 \rightarrow n$ be the environment-dependent path corresponding to $\min \limits_{\pi: 0 \rightarrow n}^1 \rho(\frac1n \mu_{\pi}, \nu)$. It is crucial to note that $\pi_n$ depends on the observed edge labels $((U_i^j)_{0 \leq i < n, 1\leq j \leq D})$ of the first $n$ trials only.

By definition of grid entropy, $\frac1n \mu_{\pi_n} \Rightarrow \nu$ a.s.. In particular, we get convergence in probability: 
\[ \lim_{n \rightarrow \infty} P \bigg( \rho \bigg(\frac1n \mu_{\pi_n} , \nu \bigg) \geq \frac{\epsilon}2 \bigg) = 0\]

Thus $\exists n \in \bN$ s.t.
\[ P \bigg( \rho \bigg(\frac1n \mu_{\pi_n} , \nu \bigg) \geq \frac{\epsilon}2 \bigg)  < \frac{\epsilon}4 \]
We claim that $\rho(E[\frac1n \mu_{\pi_n}], \nu)<\epsilon$. Let $\mathcal{E}$ be the event $\{\rho (\frac1n \mu_{\pi_n} , \nu ) \geq \frac{\epsilon}2\}$ so $P(\mathcal{E})< \frac{\epsilon}4$. We split the expectation:
\[E \bigg[ \frac1n \mu_{\pi_n} \bigg] = \int_{\mathcal{E}} \frac1n \mu_{\pi_n} dP + \int_{\mathcal{E}^C} \frac1n \mu_{\pi_n} dP \]
Since the Levy-Prokhorov metric is upper bounded by the total variation, then
\[\rho \bigg(\int_{\mathcal{E}} \frac1n \mu_{\pi_n} dP, \nu \cdot P(\mathcal{E}) \bigg) \leq \bigg( \bigg| \bigg |\frac1n\mu_{\pi_n} \bigg|\bigg|_{TV} + ||\nu||_{TV} \bigg)P(\mathcal{E}) = 2P(\mathcal{E}) < \frac{\epsilon}2 \]
On $\mathcal{E}^C$,  we have $\rho (\frac1n \mu_{\pi_n} , \nu ) < \frac{\epsilon}2$. By definition of $\rho$, for all measurable $A \in \mathcal{B}([0,1])$,
\begin{align*} 
\bigg(\int_{\mathcal{E}^C} \frac1n \mu_{\pi_n} dP\bigg) (A) &= \int_{\mathcal{E}^C} \frac1n \mu_{\pi_n} (A) dP \\
& \leq \int_{\mathcal{E}^C} \nu(A^{\frac{\epsilon}2} )+ \frac{\epsilon}2 dP \leq \nu(A^{\frac{\epsilon}2} ) P(\mathcal{E}^C) +\frac{\epsilon}2 
\end{align*}
and similarly
\begin{align*} 
 \nu(A) P(\mathcal{E}^C) &= \int_{\mathcal{E}^C} \nu(A) dP \\
 &\leq   \int_{G^C} \frac1n \mu_{\pi_n} (A^{\frac{\epsilon}2}) + \frac{\epsilon}2 dP  \leq \bigg( \int_{\mathcal{E}^C} \frac1n \mu_{\pi_n} dP\bigg) (A^{\frac{\epsilon}2}) + \frac{\epsilon}2 
\end{align*}
hence
\[ \rho \bigg(\int_{\mathcal{E}^C} \frac1n \mu_{\pi_n} dP, \nu P(\mathcal{E}^C) \bigg) \leq \frac{\epsilon}2   \]
By subadditivity of $\rho$,
\begin{align*}
 \rho \bigg( E \bigg[\frac1n \mu_{\pi_n} \bigg], \nu \bigg) &=  \rho \bigg( \int_{\mathcal{E}} \frac1n \mu_{\pi_n} dP + \int_{\mathcal{E}^C} \frac1n \mu_{\pi_n} dP, \nu P(\mathcal{E}) + \nu P(\mathcal{E}^C) \bigg) \\
 &\leq \rho \bigg( \int_{\mathcal{E}} \frac1n \mu_{\pi_n} dP, \nu P(\mathcal{E}) \bigg) + \rho\bigg(\int_{\mathcal{E}^C} \frac1n \mu_{\pi_n} dP,  \nu P(\mathcal{E}^C) \bigg) \\
 & < \epsilon 
 \end{align*}

It remains to construct a strategy $\chi$ for which $\frac1n \mu_{\pi_n} = \frac1n \mu_{0 \rightarrow n} (\chi)$ for this fixed $n$. But this is trivial since conditioned on the observed edge labels  $(U_i^j)_{0 \leq i \leq n-1, 1 \leq j \leq D}$ from the first $n$ trials, $\pi_n$ is a deterministic path $0 \rightarrow n$ (namely the one which minimizes $\rho(\frac1n \mu_{\pi}, \nu)$). 

To be concrete, consider the product measure $\chi'$ on $([0,1]^D)^{\bZ_{\geq n}} \times \{1,\ldots,D\}^{\bZ_{\geq n}}$ given by
\[ \chi' = \Lambda^{\infty} \times \delta_{(1,1,\ldots)} \]
Heuristically, $\chi'$ is a partial strategy always picking the "top" edge label $U_i^1$ for $i \geq n$.

Also let $\chi''$ be the measure on $([0,1]^D)^{\bZ_{0 \leq i < n}} \times \{1,\ldots,D\}^{\bZ_{0 \leq i < n}}$ determined by
\begin{align*} 
&\chi''(A \times \{(j_0,\ldots,j_{n-1}) \}) \\
&= \Lambda^{\infty}(\{(u_i^j)_{0 \leq i < n, 1 \leq j \leq D} \in A: \pi_n((u_i^j)_{0 \leq i < n, 1 \leq j \leq D}) = (j_0,\ldots, j_{n-1})\}) 
\end{align*}
$\forall A \in \mathcal{B}(([0,1]^D)^{\bZ_{0 \leq i < n}}), (j_i)_{0 \leq i < n} \in \{1,\ldots,D\}^{\bZ_{0 \leq i < n}}$ where $\pi_n((u_i^j)_{0 \leq i < n, 1 \leq j \leq D})$ denotes the sequence of $n$ indices $(j_0,\ldots, j_{n-1})$ corresponding to the path $\pi_n$ when the first $n$ trials yield observed labels $(u_i^j)_{0 \leq i < n, 1 \leq j \leq D}$.
That is, conditioned on the observed values from the first $n$ trials, $\chi''$ always picks the path corresponding to $\pi_n$.

\bigskip

Now consider the strategy $\chi =  \chi'' \times \chi'$.  $\mu_{0 \rightarrow n}(\chi) = \mu_{\pi_n}$  hence  $E[\frac1n \mu_{\pi_n}] = E[\frac1n \mu_{0\rightarrow n}(\chi)]$. 

This contradicts  
\[B_{\epsilon}(\nu) \cap cl \bigg\{E \bigg[\frac1n \mu_{0 \rightarrow n}(\chi) \bigg]: \chi \bigg\}  = B_{\epsilon}(\nu) \cap cl\{E[\delta_{X_0(\chi)}]: \chi\} = \emptyset\]
\end{proof}

Combining Lemmas \ref{1st}-\ref{3rd},  we get that a.s., (vi) $\subseteq$ (v) $\subseteq$ (i) $\subseteq$ (ii) $\subseteq $ cl(iii) = cl(iv)  in Theorem \ref{limitPts}. 
To complete the proof we only need to show the last  inclusion cl (iv) $\subseteq$ (vi) and the fact that (iv) is weakly closed.

\subsection{What Happens with Independent Strategies}
Let us now focus on the case of independent  strategies. Recall that these look like $\chi = \bigtimes \limits_{i=0}^{\infty} \psi_i$ where  $(\psi_i)_{i \geq 0}$ are single-edge strategies.

First, consider the even simpler case of i.i.d. strategies $\chi$, where all the $\psi_i$ are identical, resulting in i.i.d. $X_i$. Then the Glivenko-Cantelli Theorem \cite[Thm.~2.4.7]{durrett}  gives that a.s. the empirical measures $\frac1n \mu_{0 \rightarrow n}(\chi)$ converge weakly to the common law of the $X_i$'s, $\sigma_{\psi_0}$.

\begin{theorem}[Glivenko-Cantelli Theorem]\label{thm1}
Let $F_{\gamma}$ be the cumulative distribution function of $\gamma$, let $Y_i \sim \gamma$ be i.i.d.  random variables and let 
\[F_{n}(y) = \frac1{n} \sum_{i=1}^n \mathbf{1}_{\{Y_i \leq y\}} \]
be the cumulative distribution functions of the empirical measures. Then 
\[\sup_y |F_n(y) - F_{\gamma}(y)| \rightarrow 0 \ \mbox{a.s. as} \ n \rightarrow \infty \]

\end{theorem}

As an immediate corollary we get the outstanding inclusion (vi) $\supseteq $  cl (iv) mentioned at the end of Section \ref{sect31}.

\begin{corollary}

\end{corollary}
A.s. we have
\[  \bigg\{\mbox{limit pts of} \ \frac1n \mu_{0 \rightarrow n}(\chi): \ \mbox{i.i.d. strategies $\chi$}\bigg\} \supseteq cl\{\sigma_{\psi}: \mbox{single-edge strategies} \ \psi\} \]
\begin{proof}
Fix a dense subset $\mathcal{O}$ of the deterministic, weakly compact set cl (iv) s.t. $\mathcal{O} \subseteq \mbox{(iv).}$ For every $\sigma_{\psi} \in \mathcal{O}$, apply Theorem \ref{thm1} to the i.i.d. strategy $\chi = \bigtimes \limits_{i=0}^{\infty} \psi$ to get  $\frac1n \mu_{0\rightarrow n}(\chi) \Rightarrow \sigma_{\psi}$ a.s.. Therefore $\mathcal{O} \subseteq$ (vi) a.s.. The inclusion follows since (vi) is weakly closed.\\
\end{proof}

Thus we have shown that the sets (i), (ii), cl(iii), cl(iv), (v), (vi) in Theorem \ref{limitPts} are equal a.s..

Before proceeding to look further into single-edge strategies, it is worth mentioning a  version of the Glivenko-Cantelli Theorem for independent but not necessarily i.d. sequences from \cite{wellner}. It gives further insight into the limit points of the empirical measures $\frac1n \mu_{0\rightarrow n}(\chi)$ for independent strategies $\chi$.

\begin{theorem}\label{generalGlivenko}
Let $Y_i$ be a sequence of independent random variables with distributions $\sigma_i$. Define $\overline{\sigma}_n = \frac1n (\sigma_0+\ldots + \sigma_{n-1})$ to be the averages of these distributions and let \\ $\frac1n \mu_n = \frac1n (\delta_{Y_0}+\ldots+\delta_{Y_{n-1}})$  be the  empirical measures. If $\{\overline{\sigma}_n\}$ is tight then $\rho(\overline{\sigma}_n, \frac1n \mu_n) \rightarrow 0$ a.s..
\end{theorem}

In our case, $\mathcal{M}_1$ is weakly compact so any sequence of probability measures in $\mathcal{M}_1$ is tight. 

Thus, for independent strategies $\chi = \bigtimes \limits_{i=0}^{\infty} \psi_i$ with $\psi_i$ single-edge strategies, we have \\ $\rho(\overline{\sigma_n}, \frac1n \mu_{0 \rightarrow n}(\chi)) \rightarrow 0 \ \mbox{a.s.}$ where $\overline{\sigma_n}$ denote the averages of the distributions of $X_i(\chi)$:
\[\overline{\sigma}_n  = \frac1n (\sigma^0_{\chi} + \ldots +\sigma^{n-1}_{\chi}) = \frac1n (\sigma_{\psi_0} + \ldots +\sigma_{\psi_{n-1}})= \sigma_{\frac1n(\psi_0+\ldots + \psi_{n-1})} \]

 This not only confirms that cl(iv) and (v) are equal, but it tells us that a.s. the empirical measures $\frac1n \mu_{0 \rightarrow n}(\chi)$ for independent strategies $\chi$ have the exact same limit points as the distributions $\sigma_{\frac1n(\psi_0+\ldots + \psi_{n-1})}$ corresponding to the law of $X$ chosen according to the average of the single-edge strategies $\psi_i$.

\section{Single-Edge Strategies Revisited}
\subsection{Single-Edge Strategies in Terms of Conditional Probabilities}\label{vectorFcns}
From what we have shown thus far, we only need to focus on single-edge strategies, as the set of limit points of empirical measures $\frac1n \mu_{0 \rightarrow n}(\chi)$ coincides with 
\[ cl \{\sigma_{\psi}: \mbox{single-edge strategies} \ \psi\} \]

Rather than using the measure definition of single-edge strategies, it  is  more practical to work with the vector $\vec{p}$ of conditional probabilities defined as
 \[ p_k(u_1,\ldots, u_D) := P_{\psi}[J = k \mid (U^j)_{1 \leq j \leq D} = (u^j)_{1 \leq j \leq D}]  \]
 
Then $\psi$ evaluated on product sets is given by
\begin{equation}\label{PsiFormula} \psi(A \times B) = \int_A \sum_{j=1}^D p_j (u_1,\ldots, u_D) \mathbf{1}_{\{j \in B\}} \ d\Lambda^D(u_1,\ldots, u_D) \ \forall A \in \mathcal{B}([0,1]^D), B \subseteq \{1,\ldots, D\}
\end{equation}
and $\sum \limits_{j=1}^D p_j \equiv 1$ everywhere. 

Intuitively, each $p_k(u_1,\ldots, u_D)$ is the probability of choosing $u_k$ when the observed samples are $(u_1,\ldots,u_D)$. This justifies the term "single-edge strategy," because $\vec{p}$ is prescribing the strategy by which we make our choice once we have the $D$ observed samples.

Of course, the vector $\vec{p}=(p_1,\ldots, p_D)$ determines $\psi$ by  \eqref{PsiFormula}. And if all the $p_j$ are 0-1 valued $\Lambda^{\times D}$-a.e. then the single-edge strategy $\psi$ is deterministic.

Even though different $\vec{p}$ may give rise to the same $\psi$, we  conflate the two notions and call both $\psi$ and $\vec{p}$ the single-edge strategy. We use $X(\vec{p}), X(\chi)$ and $\sigma_{\vec{p}}, \sigma_{\psi}$ and $F_{\vec{p}}, F_{\psi}$ interchangeably. At the end of the day, all that  matters is whether two vectors $\vec{p}, \vec{q}$ yield the same law $\sigma_{\vec{p}}$ of $X(\vec{p})$.

We may now write the cumulative distribution function of $\sigma_{\vec{p}}$ in terms of $\vec{p}$:
\begin{equation} \label{cdf}
\begin{split}
F_{\vec{p}}(y) &=\int_{[0,1]^D} \sum_{j=1}^D p_j (u_1,\ldots, u_D) \mathbf{1}_{[0,y]}(u_j) d\Lambda^D(u_1,\ldots, u_D)  \\
&= \int_{[0,y] \times [0,1]^{D-1}}  p_1 (u_1,\ldots, u_D)  d\Lambda^D(u_1,\ldots, u_D)  \\
&+ \ldots + \int_{[0,1]^{D-1} \times [0,y] }  p_D (u_1,\ldots, u_D)  d\Lambda^D(u_1,\ldots, u_D) 
\end{split}
\end{equation}
Therefore $\sigma_{\vec{p}} \ll \Lambda$ with  density 
\begin{equation}\label{density}
\begin{split}
 f_{\vec{p}}(y) &=\int_{[0,1]^{D-1}} p_1(y,u_2,\ldots, u_D) d\Lambda^{D-1}(u_2,\ldots, u_D)  \\
 &+ \cdots  +  \int_{[0,1]^{D-1}} p_D(u_1,\ldots,u_{D-1},y) d\Lambda^{D-1}(u_1,\ldots, u_{D-1})
 \end{split}
 \end{equation}
It is clear that $f_{\vec{p}} \in [0,D]$.\\

Let us give an example that will make everything clear. Consider the deterministic single-edge strategy $\vec{p}^{MAX}$ that always chooses the maximum of the observed edge labels. In our notation,
\[ p_j^{MAX}(u_1, \ldots, u_D) = \mathbf{1}_{\{u_j \geq u_k \ \forall 1 \leq k \leq D\}}\]
The resulting distribution $\sigma_{MAX}$ has cdf
\[F_{MAX}(y)  = P(U^j \leq y \ \forall 1 \leq j \leq D) =  y^D \]
and density
\[f_{MAX}(y) = D y^{D-1}\]
This single-edge strategy  plays a crucial role later in our description of the extreme points of $\{\sigma_{\vec{p}}\}$.

We will also be interested in single-edge strategies "scrambled" by a $\Lambda$-measure-preserving bijection.

\begin{definition}
Let $\vec{p}$ be a single-edge strategy and let $\phi: [0,1] \rightarrow [0,1]$ be a $\Lambda$-measure-preserving bijection. We define $\vec{p}^\phi$ to be the single-edge strategy with coordinate functions 
\[p_j^{\phi}(u_1, \ldots, u_D) := p_j (\phi(u_1), \ldots, \phi(u_D))\]
It is clear that $\sum p_j^{\phi} \equiv 1$ still so it is a valid single-edge strategy. Furthermore, since $\phi$ is measure-preserving with respect to $\Lambda$, then by a change of variables in the integral formula for $f_{\vec{p}}$ we get 
\[ f_{\vec{p}^{\phi}} = f_{\vec{p}} \circ \phi \]
\end{definition}

The following lemma collects  some basic facts about single-edge strategies.

\begin{lemma}\label{lemma3}
Let $\vec{p}$ be a single-edge strategy.
\begin{enumerate}[label=(\roman*)]
\item For any Borel set $A \in \mathcal{B}([0,1])$, 
\[\Lambda(A)^D \leq \sigma_{\vec{p}}(A) \leq 1 - (1-\Lambda(A))^D    \]
In particular, $X(\vec{p}^{MAX})$ stochastically dominates $X(\vec{p})$ and  $\sigma_{\vec{p}} \ll \Lambda, \Lambda \ll \sigma_{\vec{p}}$.
\item We have
\[ \sup_{\phi} E[\sigma_{\vec{p}^{\phi}}] \leq E[\sigma_{MAX}]\]
where the supremum is taken over $\Lambda$-measure-preserving bijections $\phi: [0,1] \rightarrow [0,1]$.
\item Convex combinations of single-edge strategies $\vec{p}$ translate to convex combinations of $\sigma_{\vec{p}}, f_{\vec{p}}, F_{\vec{p}}$.
\end{enumerate}
\end{lemma}

\begin{remark}
It is not clear whether the supremum in (ii) is achieved, but this is beyond the scope of this paper.
\end{remark}

\begin{proof}
(i) If all $D$ edge labels $u_1, \ldots, u_D$ are in the set $A$, then so must be the edge label chosen from among them. Therefore $\Lambda(A)^D \leq \sigma_{\vec{p}}(A)$. The other inequality follows by replacing $A$ with $A^C$.

Recalling that $F_{MAX}(y) = y^D$, we get 
\[F_{MAX}(y) \leq F_{\vec{p}}(y) \ \forall y\]

(ii) For any $\Lambda$-measure-preserving bijection $\phi$, we apply the tail integral formula for expectation and use (i) to get
\[E[\sigma_{\vec{p}^{\phi}}] = \int_{0}^{1} 1-F_{\vec{p}^{\phi}}(t) dt  \leq \int_{0}^{1} 1-F_{MAX}(t) dt = E[\sigma_{MAX}]\]

(iii) For single-edge strategies $\vec{p}, \vec{q}$ and $t \in [0,1]$, $t\vec{p} + (1-t)\vec{q}$ is itself a single-edge strategy with
\[ F_{t\vec{p} + (1-t)\vec{q}} = t F_{\vec{p}} + (1-t) F_{\vec{q}}, \  f_{t\vec{p} + (1-t)\vec{q}} = t f_{\vec{p}} + (1-t) f_{\vec{q}}, \ \mu_{t\vec{p} + (1-t)\vec{q}} = t \sigma_{\vec{p}} + (1-t) \sigma_{\vec{q}}\]

\end{proof}

\subsection{"Consistent" Single-Edge Strategies}
Before proceeding further, we explain why we can restrict ourselves to single-edge strategies $\vec{p}$ with some convenient symmetries that render $\vec{p}$ consistent.

We say a permutation $\iota \in S_D$ acts on a $D$-tuple $\vec{u}$ by applying $\iota$ to the indices:
\[ \iota(\vec{u}) = (u_{\iota(1)}, \ldots, u_{\iota(D)})\]

 \begin{lemma}\label{lemma4}
 Let $\vec{p}$ be any single-edge strategy. Then there exists another single-edge strategy $\vec{q}$ that gives rise to the same distribution $\sigma_{\vec{p}} = \sigma_{\vec{q}}$ s.t. $\forall x \in [0,1], 1 \leq i \leq D$
 \begin{equation}\label{eq17}
      f_{\vec{p}}(x) = f_{\vec{q}}(x) = D \int_{[0,1]^{D-1}} q_i(u_1, u_2, \ldots, x, \ldots, u_D) d\Lambda^{D-1}(u_1,\ldots, \widehat{u_i}, \ldots, u_D)  
 \end{equation}
 where the $x$ occurs at position $i$, s.t. 
  \begin{equation}\label{eq18}
q_1(\iota_i(\vec{u})) = q_2(\iota_{i-1}(\vec{u})) = \ldots = q_D(\iota_{i+1}(\vec{u})) \ \forall 1 \leq i \leq D \ \forall \vec{u} 
 \end{equation}
 where each $\iota_j$ is the cyclic shift
\[\iota_j(\vec{u}) = (u_j, u_{j+1}, \ldots, u_D, u_1,\ldots, u_{j-1})\]
and s.t.
  \begin{equation}\label{eq19}
q_i(\vec{u}) = q_i(\iota(\vec{u})) \ \forall 1 \leq i \leq D, \forall \vec{u}, \mbox{and} \ \forall \iota \in S_D \ \mbox{with} \ \iota(i) = i
 \end{equation}
 \end{lemma}

\begin{remark}\label{RMK11}
The new single-edge strategy $\vec{q}$ is consistent across all permutations of a tuple $\vec{u}$. That is, given an unordered tuple $(u_1, \ldots, u_D)$ we can say that $\vec{q}$ chooses each $u_i$ with probability $t_i$. Then every $q_j(\iota(\vec{u}))$ for $\iota \in S_D$ picks out the probability $t_{\iota(j)}$ i.e. $q_j$ outputs the probability of choosing the $j$th entry in its input tuple.

Furthermore, both the density $f_{\vec{q}}$ and the entire single-edge strategy $\vec{q}$ are uniquely determined by any one of the $q_i$. That is, given a measurable function $q_i: [0,1]^D \rightarrow [0,1]$ whose integral over $[0,1]^D$ with respect to the product measure $\Lambda^{\times D}$ is $\frac1D$, which is invariant under permutations in $S_D$ fixing $i$, and which satisfies
\[ \sum_j q_i(\iota_j(\vec{u})) = 1 \ \forall \vec{u} \]
we can use cyclic shifts to define a valid corresponding single-edge strategy $\vec{q}$ (that satisfies $\sum q_j \equiv 1$ and $\int_{[0,1]^D} q_j \ d\Lambda^{\times D} = \frac1D$) and we can compute $f_{\vec{q}}$ directly from $q_i$.
\end{remark}

\begin{remark}
It is easy to check that $\vec{p}^{MAX}$ satisfies \eqref{eq17}-\eqref{eq19}. Furthermore, if $\vec{p}$ is a consistent single-edge strategy and $\phi:[0,1] \rightarrow [0,1]$ is a $\Lambda$-measure-preserving bijection then so is $\vec{p}^{\phi}$.
\end{remark}

 \begin{proof}
 The intuition is that we take the average of the original $p_i$ over the desired symmetries. We do this in two steps. For each $1 \leq i \leq D$ define
\[p_i'(\vec{u}) := \frac{p_1(\iota_i(\vec{u})) + p_2(\iota_{i-1}(\vec{u})) + \ldots + p_i(\iota_1(\vec{u}))+  \ldots p_D(\iota_{i+1}(\vec{u}))}{D}\]
\[q_i(\vec{u}) = \frac1{(D-1)!} \sum_{\iota \in S_{D}, \iota(i) = i} p_i'(u_{\iota(1)}, u_{\iota(2)} \ldots, u_{\iota(D)})\]
A straightforward computation shows that $\vec{q}$ satisfies the required properties \eqref{eq17}-\eqref{eq19}.
\end{proof}

For the rest of the paper we restrict ourselves to these consistent single-edge strategies, which will simplify our computations.

\subsection{Closure of \texorpdfstring{$\{\sigma_{\vec{p}}\}$}{mP}}
The last remaining part of Theorem \ref{limitPts} is to show that the set of distributions $\sigma_{\vec{p}}$ is weakly closed. 





\begin{theorem}\label{closure}
$\{\sigma_{\vec{p}}:   \mbox{ single-edge strategies} \ \vec{p}\}$ is weakly closed
\end{theorem}

\begin{proof}
Suppose $\sigma_{\vec{p}^{n}} \Rightarrow \xi$ for some consistent single-edge strategies $\vec{p}^n$. We seek a single-edge strategy yielding the distribution $\xi$.

Consider any $n$. Let $\nu_{\vec{p}^n}$ be the distribution on $[0,1]^D$ given by integration against \\ $D p_1^n d\Lambda^{\times D}$. Observe that \eqref{eq17} implies
\[  f_{\vec{p}^n}(y) =  \int_{[0,1]^{D-1}} D p_1^n(y, u_2, \ldots, u_D) d\Lambda^{D-1}(u_2, \ldots, u_D)   \]
so $\sigma_{\vec{p}^n}$ is just the first coordinate marginal of $\nu_{\vec{p}^n}$.

Compactness yields a weakly convergent subsequence $\nu_{\vec{p}^{n_j}} \Rightarrow \nu$. By a standard argument, since $Dp_1^n$ are uniformly bounded by $D$, then $\nu \ll \Lambda^{\times D}$ and has a density of the form $Dp_1$ for some measurable function $p_1: [0,1]^D \rightarrow [0,1]$. This along with the fact that $\int_{[0,1]^D} p_1 d\Lambda^{\times D} = \frac1D$ are enough for $p_1$ to give rise to a single-edge strategy $\vec{p}$ by the first Remark following Lemma \ref{lemma4}. 

But then the corresponding first coordinate marginals must also converge weakly, so we get $\sigma_{\vec{p}^{n_j}} \Rightarrow \xi'$ where $\xi'$ is the corresponding marginal of $\nu$:
\[ \xi'(A) =\int_{A\times [0,1]^{D-1}} Dp_1 d\Lambda^D = \sigma_{\vec{p}}(A) \ \forall A \in \mathcal{B}([0,1])\]
Thus $\xi' = \sigma_{\vec{p}}$. On the other hand, $\sigma_{\vec{p}^n} \Rightarrow \xi$ so by uniqueness of weak limits, $\xi = \xi' = \sigma_{\vec{p}}$.
\end{proof}

Therefore the set of possible distributions $\{\sigma_{\vec{p}}\}$ of $X(\vec{p})$ over single-edge strategies $\vec{p}$ almost surely coincides with the set of limit points of the empirical measures $\frac1n \mu_{0 \rightarrow n}(\chi)$ over all strategies $\chi$. In particular these have the same extreme points.

\section{Extreme Points of \texorpdfstring{$\{\sigma_{\vec{p}}\}$}{mp}}\label{extreme}
Our goal in this section is to characterize the extreme points of the possible distributions of $X$ we can observe as we vary the underlying  single-edge strategy. It turns out that the extreme points of $\{\sigma_{\vec{p}}:  \mbox{ single-edge strategies} \ \vec{p}\}$ are precisely those $\sigma_{\vec{p}}$ which have scramblings with mean converging to the mean of $\sigma_{MAX}$, or, equivalently, those $\sigma_{\vec{p}}$ with deterministic $\vec{p}$.

 \begin{theorem}\label{thm4}
Let $\vec{p}$ be a single-edge strategy.  The following are equivalent:
\begin{enumerate}[label=(\roman*)]
\item $\sigma_{\vec{p}}$ is an extreme point
\item Any consistent single-edge strategy achieving $\sigma_{\vec{p}}$  must be deterministic
\item $f_{\vec{p}}$ is not constant on sets of positive $\Lambda$-measure and $\sigma_{\vec{p}}$ is given by the following single-edge strategy $\vec{q}$:
\[ q_k(u_1,\ldots, u_D) = \textbf{1}_{\{f_{\vec{p}}(u_k) \geq f_{\vec{p}}(u_i) \ \forall i\}} \ \mbox{for $\Lambda^{\times D}$-a.a.} \ (u_1,\ldots, u_D) \]
In other words, $\sigma_{\vec{p}}$ is achieved by  the deterministic single-edge strategy "choose \\ whichever label yields a higher value when evaluating the density $f_{\vec{p}}$"
\item $\sup_{\phi} E[\sigma_{\vec{p}^{\phi}}] = E[\sigma_{MAX}]$
\item For $U \sim [0,1]$ if $f_{\vec{p}}(U), f_{MAX}(U)$ as $[0,D]$-valued random variables  on the probability space $([0,1], \mathcal{B}([0,1]), \Lambda)$, then they have the same distribution. That is,
\[P(f_{\vec{p}} \leq x) = P(f_{MAX} \leq x) \ \forall x\]
\end{enumerate}
\end{theorem}

\begin{remark}
Since $f_{MAX}(y) = Dy^{D-1}$ on $[0,1]$ then $\frac1D f_{MAX} \sim \mbox{Beta}(D,1)$. Thus (v) implies that for $U\sim \mbox{Unif}[0,1]$, $f_{\vec{p}}(U)$ has a continuous probability distribution. 
\end{remark}

We prove this theorem over the next couple of sections.

\subsection{Deterministic Single-Edge Strategies }

We begin by characterizing the extreme points of the set of distributions $\{\sigma_{\vec{p}}\}$ in terms of deterministic single-edge strategies.

 \begin{lemma}\label{lemma6}
(i) $\Leftrightarrow$ (ii) in Theorem \ref{thm4}
\end{lemma}

\begin{proof}  We prove both contrapositives. 

First, consider a distribution $\sigma_{\vec{p}}$ achieved by a non-deterministic consistent single-edge strategy $\vec{p}$, i.e. a single-edge strategy satisfying \eqref{eq17}-\eqref{eq19}.  We will construct a set of positive $\Lambda^{\times D}$ measure on which we  perturb $p_1$ in a way that allows us to write $\vec{p}$ as a non-trivial average of two single-edge strategies $\vec{q}, \vec{r}$.

By our assumption, there exists $0 < \epsilon < \frac12$ s.t. the set
\[A = \{\vec{u}: p_1(\vec{u}) \in (\epsilon, 1-\epsilon), p_1(\vec{u}) \geq p_1(\iota_i(\vec{u})) \ \forall 1 \leq i \leq D\} \ \mbox{has positive $\Lambda^{\times D}$ measure}\]
where $\iota_i$ are the cyclic shifts as before. Then there are $a_1 < b_1, a_2 < b_2$ s.t. $[a_1, b_1], [a_2, b_2]$ are disjoint and 
\[A' = A \cap [a_1, b_1] \times [a_2, b_2]^{D-1} \ \mbox{has positive $\Lambda^{\times D}$ measure}\]
The purpose of this is to ensure that the first coordinate in a tuple in $A'$ cannot appear in any other position in another tuple in $A'$.

By \eqref{eq18}, the fact that $p_1(\vec{u}) < 1-\epsilon$ on $A'$, and the definition of single-edge strategies,
\[\sum_{k=2}^D p_1(\iota_k(\vec{u})) = \sum_{k=2}^D p_k(\vec{u}) = 1-p_1(\vec{u}) > \epsilon \ \forall \vec{u} \in A'\]
Hence for any $\vec{u} \in A'$, there exists $2 \leq k \leq D$ s.t. $p_1(\iota_k(\vec{u})) > \frac{\epsilon}D$. It follows that there is $2 \leq k \leq D$ s.t.
\[A'' = \bigg\{\vec{u} \in A': p_1(\iota_k(\vec{u})) \in \bigg(\frac{\epsilon}D, 1-\epsilon\bigg)\bigg\} \ \mbox{has positive $\Lambda^{\times D}$ measure}\]
Define
\[q_1 := p_1 + \frac{\epsilon}{D!} \sum_{\iota \in S_D, \iota(1)= 1} (  \mathbf{1}_{\iota(A'')} - \mathbf{1}_{\iota(\iota_k^{-1}(A''))}), r_1 := p_1 + \frac{\epsilon}{D!} \sum_{\iota \in S_D, \iota(1)= 1} ( - \mathbf{1}_{\iota(A'')} + \mathbf{1}_{\iota(\iota_k^{-1}(A''))})\]
It is clear that $q_1, r_1$ average to $p_1$, and are both invariant under permutations in $S_D$ fixing 1. Since $p_1(\vec{u}) \in (\frac{\epsilon}D, 1-\epsilon) \ \forall \vec{u} \in \iota(A'') \cup \iota(\iota_k^{-1}(A'')) \ \forall \iota \in S_D, \iota(1) = 1$ then both $q_1, r_1$ have ranges in $[0,1]$. Also, note that 
\[\int_{[0,1]^D}  \mathbf{1}_{\iota(B)} d\Lambda^{\times D} = \Lambda^{\times D}(B) \ \forall B \in \mathcal{B}([0,1]^D), \iota \in S_D \]
so it follows that 
\[\int_{[0,1]^D} q_1 d\Lambda^{\times D} = \int_{[0,1]^D} r_1 d\Lambda^{\times D} = \int_{[0,1]^D} p_1 d\Lambda^{\times D} = \frac1D \]
Finally, note that by a simple counting argument, for any $\vec{u}$,
\[\sum_i \sum_{\iota \in S_D, \iota(1)= 1}  \mathbf{1}_{\iota(A'')}(\iota_i(\vec{u})) = \sum_{\iota \in S_D} \mathbf{1}_{\iota(A'')}(\vec{u}) = \sum_i \sum_{\iota \in S_D, \iota(1)= 1}  \mathbf{1}_{\iota(\iota_k^{-1}(A''))}(\iota_i(\vec{u}))  \]
so
\[\sum_i q_1(\iota_i(\vec{u})) = \sum_i r_1(\iota_i(\vec{u})) = \sum_i  p_1(\iota_i(\vec{u})) = 1 \ \forall \vec{u}\]
Thus $q_1, r_1$ are valid probability functions which average to $p_1$. By the remark after Lemma \ref{lemma4}, these uniquely determine single-edge strategies $\vec{q}, \vec{r}$ whose average is $\vec{p}$. Thus $f_{\vec{p}}$ is the average of $f_{\vec{q}}, f_{\vec{r}}$. It remains to show this convex combination is non-trivial.

 Consider any $\vec{u} = (u_1, u_2,\ldots, u_D)$ with $u_1 \in [a_1, b_1]$. For any $\iota \in S_D, \iota(1) = 1$,
 \[\iota^{-1}(k) \neq 1 \Rightarrow u_{\iota^{-1}(k)} \in [a_2, b_2] \Rightarrow u_{\iota^{-1}(k)} \notin [a_1, b_1] \Rightarrow \vec{u} \notin \iota(\iota_k^{-1}(A''))\]
It follows that 
\[  q_1(\vec{u}) = p_1(\vec{u}) + \frac{\epsilon}{D!} \sum_{\iota \in S_D, \iota(1)= 1}  \mathbf{1}_{\iota(A'')}(\vec{u}) \geq p_1(\vec{u}) + \frac{\epsilon}{D!} \mathbf{1}_{A''}(\vec{u})\]
so by \eqref{eq17},
\[ f_{\vec{q}}(u_1) \geq f_{\vec{p}}(u_1) + \frac{\epsilon}{(D-1)!}\int_{[0,1]^{D-1}} \mathbf{1}_{A''} (u_1, u_2, \ldots, u_D) \ d\Lambda^{\times (D-1)}(u_2,\ldots, u_D) \]
Since $\Lambda^{\times D} (A'') > 0$ then there is $\delta > 0$ and a Borel set $B \subset [a_1, b_1]$ with $\Lambda(B)> 0$ s.t. the integral above is $\geq \delta$ for $u_1 \in B$. It follows that $f_{\vec{q}} \geq f_{\vec{p}} + \frac{\epsilon}{(D-1)!}\delta$ on $B$ so the convex combination is non-trivial. Therefore $\sigma_{\vec{p}}$ is not an extreme point.

Now suppose $\sigma_{\vec{p}}$ is not an extreme point, say $\sigma_{\vec{p}} = t\sigma_{\vec{q}} + (1-t) \sigma_{\vec{r}}$ for some consistent single-edge strategies $\vec{q},  \vec{r}$ with $\sigma_{\vec{q}} \neq \sigma_{\vec{r}}$. Then $\sigma_{\vec{p}}$ is achieved by the single-edge strategy $t \vec{q} + (1-t) \vec{r}$, which trivially satisfies \eqref{eq17}-\eqref{eq19} so is consistent. Since $\vec{q},  \vec{r}$ have range in $[0,1]$ and differ on some set of positive $\Lambda^{\times D}$-measure then on this same set $t \vec{q} + (1-t) \vec{r}$ is not $\{0,1\}$-valued. Thus $\sigma_{\vec{p}}$ is achieved by a non-deterministic single-edge strategy satisfying \eqref{eq17}-\eqref{eq19}.
\end{proof}

\subsection{"Weight Tuples" of Single-Edge Strategies}


$\sigma_{\vec{p}}$ being an extreme point will give us further properties relating to $(D+1)$-tuples \\ $(u_1, \ldots, u_{D+1})$, but these properties hold a.s. and we must be extra careful about which tuples living in measure 0 sets we are excluding. We need some setup to address this issue.

 Let $\vec{p}$ be a consistent single-edge strategy s.t. $\sigma_{\vec{p}}$ is an extreme point. By Lemma \ref{lemma6}, $p_1$ is 0-1 valued $\Lambda^{\times D}$-a.e.. Also, the set of tuples in $[0,1]^D$ with two or more duplicates has  $\Lambda^{\times D}$-measure 0. Define
 \[ S := \{\mbox{ordered} \ (u_1,\ldots,u_D) \in \bR^D \ \mbox{ duplicate-free s.t.} \ p_1(u_1,\ldots,u_D) \in \{0,1\}\} \]
 to be the set of "good" ordered $D$-tuples  we consider. Note that $\Lambda^D(S)= 1$. Once we prove more a.s. properties of the elements of $S$, we will amend this definition of $S$ accordingly. 

Combining the fact that $p_1$ is 0-1 valued on $S$ and the remark after Lemma \ref{lemma4}, we see that $p_1$ is equivalent to a choice function taking in \emph{unordered} tuples $(u_1, \ldots, u_D)$ and outputting one of the coordinates $u_j$; then for any $\iota \in S_D$, $p_1(\iota(\vec{u})) = \mathbf{1}_{\{j = \iota(1)\}}$ i.e. $p_1$ returns whether or not the first coordinate of the input ordered tuple is the choice associated with the corresponding unordered tuple.

Let 
\[
  T = \left\{  \begin{array}{l}
    \mbox{ordered} \ (u_1,\ldots,u_{D+1}) \in [0,1]^{D+1} \ \mbox{ s.t.} \\
    \mbox{all $(D+1)!$ ordered subtuples of size $D$ are in $S$} 
  \end{array}\right\}
\]
It is clear that $\Lambda^{\times (D+1)}(T) = 1$ and the tuples in $T$ contain no duplicate entries. $T$ is the set of "good" ordered $(D+1)$-tuples we restrict ourselves to.

For ordered $(D+1)$-tuples in $T$, we wish to study the possible sets of choices we can make for the $(D+1)!$ ordered subtuples of size $D$. Since the choice for some $\vec{u} \in [0,1]^D$ is also the choice for $\iota(\vec{u})$ for all $\iota \in S_D$, then rather than keeping a factor of $D!$ everywhere, we instead consider the possible sets of choices we can make for the $D+1$ \emph{unordered} subtuples of size $D$. We can encode these choices as an ordered "weight tuple" $(w_1, \ldots, w_{D+1})$ where $w_j$ is the number of times $u_j$ is the choice made. In terms of the consistent single-edge strategy $\vec{p}$,
\begin{equation}\label{equation4} 
w_j = w_j(u_1,\ldots, u_{D+1}) := \sum_{1 \leq k \leq D+1, k \neq j} p_1(u_j,u_1,\ldots, \widehat{u_j},\ldots,\widehat{u_k}, \ldots, u_{D+1})
\end{equation}
For example, the weight tuple $(D, 1, 0, \ldots, 0)$ corresponds to the choices
\[ (u_1,\ldots, \widehat{u_k}, \ldots, u_{D+1}) \mapsto u_1 \ \forall 2 \leq k \leq D+1 \ \mbox{and} \ (u_2,\ldots, u_{D+1})\mapsto u_2\]
 Note that each $u_j$ appears in exactly $D$ of the $D+1$ $D$-subtuples and there is exactly one choice for each subtuple so
\begin{equation}\label{eq20}
    w_i \in \{0, 1,\ldots, D\} \ \forall 1 \leq i \leq D+1, \sum w_i \equiv D+1 
\end{equation}
 Even though a weight tuple may correspond to more than one set of choices for the $D+1$ $D$-subtuples, it is easily checked that every valid weight tuple (satisfying \eqref{eq20}) corresponds to at least one set of choices.

The following flow chart summarizes this process:
\begin{center}
\begin{tikzpicture}
\node[draw] at (0,100pt) {Start with ordered $(D+1)$-tuple $(u_1,\ldots,u_{D+1}) \in T$};
\node[draw, align=left] at (0,50pt) {Consider the choice for each unordered $D$-subtuple $(u_1,\ldots, \widehat{u_k}, \ldots, u_{D+1})$. \\ Count occurrences of each $u_j$ among these choices.};
\node[draw] at (0,0pt) {Get an ordered $(D+1$)-weight tuple $(w_1, \ldots, w_{D+1})$ satisfying \eqref{eq20}};
\draw [->] (0,90pt) to (0,68pt);
\draw [->] (0,33pt) to (0,10pt);
\end{tikzpicture}
\end{center}

One last observation we make about the weight tuples is that they can be used to compute the density $f_{\vec{p}}$ directly. Fix any $1 \leq j \leq D+1$. Observe that \eqref{eq17} with a change of variables gives
\begin{equation*} 
\begin{split} f_{\vec{p}}(t) &= D \int_{[0,1]^{D-1}} p_1(t,u_2, \ldots, u_D) d\Lambda^{D-1}(u_2, \ldots, u_D) \\
&=  \sum_{1 \leq k \leq D+1, k \neq j} \int_{[0,1]^{D-1}}  p_1(t,u_1,\ldots, \widehat{u_j},\ldots,\widehat{u_k}, \ldots, u_{D+1}) \\
&\qquad \qquad \qquad d\Lambda^{D-1}(u_1, \ldots,  \widehat{u_j}, \ldots,\widehat{u_k}, \ldots, u_{D+1}) 
\end{split}
\end{equation*}
But $\Lambda$ is a probability measure so we can integrate everything against $du_k$ without changing the value:
\begin{equation} \label{eqn6} 
\begin{split} 
f_{\vec{p}}(t) &=   \sum_{1 \leq k \leq D+1, k \neq j} \int_{[0,1]^{D}}  p_1(t,u_1,\ldots, \widehat{u_j},\ldots,\widehat{u_k}, \ldots, u_{D+1}) d\Lambda^{D}(u_1, \ldots,  \widehat{u_j}, \ldots, u_{D+1}) \\
&=\int_{[0,1]^{D}} w_j(u_1,\ldots, t, \ldots, u_{D+1} ) d\Lambda^{D}(u_1, \ldots,  \widehat{u_j}, \ldots, u_{D+1}) \ \mbox{by  \eqref{equation4}}
\end{split}
\end{equation}
where $t$ occurs in position $j$. Note that the above formula holds for any $t$ for which
\[\{(u_1,\ldots, u_{j-1}, u_{j+1},\ldots, u_{D+1}) | (u_1,\ldots, u_{j-1},t,u_{j+1},\ldots, u_{D+1}) \in T\} \ \mbox{has} \ \Lambda^{\times D}-\mbox{measure 1}\]
i.e. it holds for $t$ in a set of $\Lambda$-measure 1.

\subsection{"Weight Tuples" of the Extreme Points}

 We claim that  $\sigma_{\vec{p}}$ being an extreme point implies that $\Lambda^{\times (D+1)}$-almost all weight tuples must be a permutation of $(D,1,0,\ldots, 0)$. That means that almost all $(D+1)$-tuples in $T$ have one coordinate that "dominates" the others in terms of the choice function, and this behaviour will imply precisely that $f_{\vec{p}}$ can be approximated by "scrambles" of $f_{MAX}$.

\begin{lemma}
Suppose $\vec{p}$ is a consistent, single-edge strategy s.t. $\sigma_{\vec{p}}$ is an extreme point. For $\Lambda^{\times (D+1)}$-a.a. ordered tuples in $T$, the corresponding weight tuple is a permutation of $(D,1,0,\ldots,0)$.
\end{lemma}

\begin{proof}
We sketch the details of the proof in the case $D \geq 3$. The case $D = 2$ is similar.

Suppose the claim is false, say the set $U$ of ordered $(D+1)$-tuples $(u_1,\ldots, u_{D+1}) \in T$ whose weight tuple $(w_1,\ldots, w_{D+1})$ is not a permutation of $(D,1,0,\ldots,0)$ has positive $\Lambda^{\times (D+1)}$-measure. Note that $U$ is clearly invariant under permutations in $S_{D+1}$. 

We follow a similar approach as in the proof of Lemma \ref{lemma6}, in that we seek write $\vec{p}$ as a convex combination of two new single-edge strategies $\vec{q}, \vec{r}$ obtained by perturbing $\vec{p}$.

First, we further restrict $U$. There exist $a_i < b_i$ s.t. $[a_i, b_i]$ are pairwise disjoint and s.t.
 \[U' := U \cap [a_1, b_1] \times \cdots \times [a_{D+1}, b_{D+1}] \ \mbox{has positive $\Lambda^{\times (D+1)}$ measure}  \]
 In this way, each coordinate in a tuple in $U'$ cannot  appear in any other position in another tuple in $U'$. 

Now consider any $(u_1, \ldots, u_{D+1}) \in U'$ so it has a maximal weight $1 \leq w_K \leq D-1$ in its weight tuple. If $w_K \geq 2$ then either there exists $i \neq K$ s.t. $w_i \geq 2$ or there are at least two $i \neq K$ s.t. $w_i = 1$ (this is because $\sum w_m = D+1$ and $w_K \leq D-1$). In either case, we can pick $I \neq K$ so that the choice for $(u_1, \ldots, \widehat{u_K}, \ldots, u_{D+1})$ is not $u_I$, hence there is another $D$-subtuple containing $u_K$ for which the choice is $u_I$. Furthermore, since $w_K \geq 2$, then there also is a $D$-subtuple containing $u_I$ for which the choice is $u_K$. On the other hand, if $w_K = 1$ then $w_m = 1 \ \forall m$. If $(u_{m_1}, \ldots, u_{m_D})$ is the $D$-subtuple for which the choice is $u_K$ then of the remaining $D-1 \geq 2$ coordinates $u_{m_{n}}$ with $m_n \neq K$  at least one, call it $u_I$, is the choice for a $D$-subtuple other than $(u_1, \ldots, \widehat{u_K}, \ldots, u_{D+1})$, i.e. for a $D$-subtuple containing $u_K$. 
 
 Since this holds for all $\vec{u} \in U'$ then there exist distinct  $1 \leq I, J, K, L \leq D+1$  s.t.
 \[  U'' := \left\{ \vec{u} \in U'\ \middle\vert \begin{array}{l}
    1 \leq w_I, \ \mbox{and} \ w_m \leq w_K \ \forall m, \ \mbox{and} \\
    \mbox{ the choice for the $D$-subtuple $(u_1,\ldots, \widehat{u_J}, \ldots, u_{D+1})$ is $u_I$, and}  \\
    \mbox{ the choice for the $D$-subtuple $(u_1,\ldots, \widehat{u_L}, \ldots, u_{D+1})$ is $u_K$} 
  \end{array}\right\}
 \]
 has positive $\Lambda^{\times (D+1)}$ measure. 
 
 The idea is that we write  the weight tuples for $(D+1)$-tuples in $U''$ as averages of two different weight tuples, and use these new weight tuples to obtain two new non-trivial, valid single-edge strategies whose densities average to $f_{\vec{p}}$.
 
 Consider any $\vec{u} \in U''$. We can write the corresponding weight tuple $(w_1, \ldots, w_{D+1})$ as an average of two other weight tuples:
 \[(\ldots, w_I, \ldots, w_K, \ldots) = \frac12 ( \ldots, w_I+1, \ldots, w_K-1, \ldots) + \frac12 (\ldots, w_I-1, \ldots, w_K+1, \ldots) \]
 The weight tuple $( \ldots, w_I+1, \ldots, w_K-1, \ldots)$ can be achieved simply by changing the choice of $(u_1, \ldots, \widehat{u_L}, \ldots, u_{D+1})$ from $u_K$ to $u_I$, whereas the weight tuple $( \ldots, w_I-1, \ldots, w_K+1, \ldots)$ can be achieved  by changing the choice of $(u_1, \ldots, \widehat{u_J}, \ldots, u_{D+1})$ from $u_I$ to $u_K$. It is trivial to see that these two new tuples are also valid weight tuples that satisfy \eqref{eq20} so they correspond to new single-edge strategies $\vec{q}, \vec{r}$ respectively. By \eqref{eqn6},  we see that $f_{\vec{p}}(y)$ is the average of $f_{\vec{q}}(y), f_{\vec{r}}(y)$.  
 
 The only difference in the $D=2$ case is that the new weight tuples are achieved by making two changes to the choice function rather than one.
 
 It remains to check that this average is non-trivial. Let 
 \[V := \{u_K: \exists u_1, \ldots, \widehat{u_K}, \ldots, u_{D+1} \ \mbox{s.t.} \ (u_1,\ldots, u_{D+1}) \in U''\} \]
and for $x \in [0,1]$ let 
\[W^{x} := \{(u_1, \ldots, \widehat{u_K}, \ldots, u_{D+1}): \ (u_1,\ldots, , u_{K-1}, x, u_{K+1}, \ldots, u_{D+1}) \in U''\} \]
 Then $\Lambda^{(D+1) \times}(U'') > 0$ implies $\Lambda(V) > 0$ and for $\Lambda$-a.a. $u_K \in V$, $\Lambda^{\times D} (W^{u_K}) > 0$. 
 
 Let us consider how the function $w_K(\cdot, \ldots, \cdot, u_K, \cdot, \ldots, \cdot): [0,1]^D \rightarrow \{0,1,\ldots, D\}$ changes from $\vec{p}$ to $\vec{r}$ for any given fixed $u_K \in V$. On $W^{u_K}$, $w_K(\cdot, \ldots, \cdot, u_K, \cdot, \ldots, \cdot)$ increases by 1 by construction of $\vec{r}$. Note that $u_K$ can only appear in the $K$th coordinate of a tuple in $U''$ (by construction of $U''$) hence $w_K(\cdot, \ldots, \cdot, u_K, \cdot, \ldots, \cdot)$ either remains unchanged or increases by 1 off of $W^{u_K}$ (because the only way it decreases by 1 is if $u_K$ had appeared at index $I$ in a tuple in $U''$). Thus, for any $u_K \in V$,
 \[f_{\vec{r}}(u_K) \geq f_{\vec{p}}(u_K) +  \Lambda^{\times D}(W^{u_K})   \]
 by \eqref{eqn6}. It follows that $\sigma_{\vec{r}} \neq \sigma_{\vec{p}}$ so this is indeed a non-trivial convex combination. This contradicts the assumption that $\sigma_{\vec{p}}$ was an extreme point. Therefore $\Lambda^{\times (D+1)}$-a.a. weight tuples must be a permutation of $(D, 1,0,\ldots, 0)$.
 \end{proof}

  We amend our definition of "good" $(D+1)$-tuples:
  \[T' := \{(u_1,\ldots, u_{D+1}) \in T: (w_1,\ldots, w_{D+1}) \ \mbox{is a perm. of } \ (D,1,0,\ldots,0)\}\]
  We still have $\Lambda^{\times (D+1)}(T') = 1$ (provided of course that $\sigma_{\vec{p}}$ is an extreme point).
  
  We now prove the (i) $\Rightarrow$ (ii) direction of Theorem \ref{thm4}. We split the proof into several claims.
  
 \begin{lemma}\label{lemma8}
 Let $\sigma_{\vec{p}}$ be an extreme point.\\
 (i) Suppose $(u_1, \ldots, u_{D+1}) \in T'$ and $u_i$ has weight $D$ in $(u_1, \ldots, u_{D+1})$.  Then for any $1 \leq j \leq D+1$, $j \neq i$ and any $u_j'$ s.t. $(u_1, \ldots, u_{j-1}, u_j', u_{j+1}, \ldots, u_{D+1}) \in T'$, either $u_i$ or $u_j'$ has weight $D$ in $(u_1, \ldots, u_{j-1}, u_j', u_{j+1}, \ldots, u_{D+1})$.

 (ii) Suppose $(u_1, u_1', u_2, \ldots, u_D) \in T'$. Then in the corresponding weight tuple, $u_1'$ has weight $D$ and $u_1$ has weight 0 if and only if
 \[p_1(u_1, u_2, u_3, \ldots, u_D) < p_1(u_1', u_2, u_3, \ldots, u_D) \]

(iii) For $\Lambda^{\times 2}$-a.a. $(u_1, u_1')$, if 
 \[p_1(u_1, u_2, u_3, \ldots, u_D) < p_1(u_1', u_2, u_3, \ldots, u_D) \]
on a set of positive $\Lambda^{\times (D-1)}$ measure then 
 \[p_1(u_1, u_2, u_3, \ldots, u_D) \leq  p_1(u_1', u_2, u_3, \ldots, u_D) \]
 for $\Lambda^{\times (D-1)}$-a.a. $(u_2, \ldots, u_D)$. 
 
(iv) For $\Lambda^{\times 2}$-a.a. $(u_1, u_1')$, if $f_{\vec{p}}(u_1) \leq f_{\vec{p}}(u_1')$ then 
 \[p_1(u_1, u_2, u_3, \ldots, u_D) \leq  p_1(u_1', u_2, u_3, \ldots, u_D) \]
 for $\Lambda^{\times (D-1)}$-a.a. $(u_2, \ldots, u_D)$.

(v) $f_{\vec{p}}$ is not constant on any set of positive $\Lambda$ measure.

(vi) For $\Lambda^{\times D}$-a.a. $(u_1, \ldots, u_D)$,
\[ p_1(u_1,\ldots, u_D) \geq \mathbf{1}_{\{f_{\vec{p}}(u_1) \geq f_{\vec{p}}(u_i) \ \forall i\}} \]

(vii) For $\Lambda^{\times D}$-a.a. $(u_1, \ldots, u_D)$,
\[ p_1(u_1,\ldots, u_D) = \mathbf{1}_{\{f_{\vec{p}}(u_1) \geq f_{\vec{p}}(u_i) \ \forall i\}} \]

 \end{lemma}

 \begin{remark}
We give heuristics to aid in understanding these claims:\\
(i) If $u_i$ is the dominant choice in $(u_1,\ldots, u_{D+1})$ then $u_i$ can never be dominated by one of the $u_j$ for $i \neq j$ in tuples containing $u_i$ and $u_j$.\\
(ii) Dominance can be determined by evaluating $p_1$.\\
(iii) If $u_1'$ dominates $u_1$ once then $u_1'$ almost always dominate $u_1$.\\
(iv) The ordering imposed by domination coincides with the ordering imposed by the values of $f_{\vec{p}}$.\\
(vii) $u_1$ is the choice in $(u_1,\ldots,u_D)$ if and only of $u_1$ maximizes the value of $f_{\vec{p}}$.
\end{remark}
 
\begin{proof}
 (i) Let $1 \leq i \leq D$ be the index for which $u_i$ has weight $D$ in $(u_1,\ldots, u_{D+1})$. Recall that this means we have the choices 
 \[ (u_1,\ldots, \widehat{u_{\ell}},\ldots, u_{D+1}) \mapsto u_i \ \forall \ell \neq i\]
Taking $\ell = j$, this means $u_i$ has weight 1 or $D$ in $(u_1,\ldots, u_{j-1}, u_j', u_{j+1},\ldots,u_{D+1})$ (because this $(D+1)$-tuple is in $T'$). If it is $D$, we are done. If it is 1, the fact that we have the choice 
 \[(u_1,\ldots, u_{j-1}, u_{j+1},\ldots, u_{D+1}) \mapsto u_i \]
 implies no $u_k, k \neq j$ can have weight $D$ so $u_j'$ must have weight $D$ in \\ $(u_1,\ldots, u_{j-1}, u_j', u_{j+1},\ldots,u_{D+1})$.
 
(ii) Recall that $p_1$ is 0-1 valued on $D$-subtuples of tuples in $T'$. We have the following sequence of if-and-only-if statements:
 \begin{align*}
     & p_1(u_1, u_2, u_3, \ldots, u_D) < p_1(u_1', u_2, u_3, \ldots, u_D) \\
     &\Leftrightarrow p_1(u_1, u_2, u_3, \ldots, u_D) = 0, p_1(u_1', u_2, u_3, \ldots, u_D) = 1  \\
     & \Leftrightarrow (u_1, u_2, u_3, \ldots, u_D) \mapsto u_i \ \mbox{for some} \ 2 \leq i \leq D \ \mbox{and} \  (u_1', u_2, u_3, \ldots, u_D) \mapsto u_1'\\
 \end{align*}
Now if $u_1'$ has weight $D$ and $u_1$ has weight 0 in $(u_1, u_1',u_2,\ldots, u_D$), then we must have
\begin{equation} \label{equation6}
 (u_1, u_2, u_3, \ldots, u_D) \mapsto u_i \ \mbox{for some} \ 2 \leq i \leq D \ \mbox{and} \  (u_1', u_2, u_3, \ldots, u_D) \mapsto u_1'
\end{equation}

 On the other hand, if we know \eqref{equation6} then the weight of $u_i$ cannot be $D$ since the choice for $(u_1', u_2, \ldots, u_i, \ldots, u_D)$ is $u_1'$. Thus $u_i$ has weight 1, so $u_1'$ has weight $D$ and $u_1$ has weight 0.

(iii) Consider $u_1, u_1'$ s.t. $(u_1, u_1', u_2,\ldots, u_D) \in T'$ for $\Lambda^{\times (D-1)}$-a.a. $(u_2, \ldots, u_D)$. This clearly holds for $\Lambda^{\times 2}$-a.a. $(u_1, u_1')$ with $u_1 \neq u_1'$. Let
\[U = \{(u_2, \ldots, u_D): (u_1, u_1', u_2,\ldots, u_D) \in T', p_1(u_1, u_2, u_3, \ldots, u_D) < p_1(u_1', u_2, u_3, \ldots, u_D) \}\]
so $\Lambda^{\times (D-1)}(U) > 0$ by assumption. 

Consider any $(u_2, \ldots, u_D) \in U$ and $u_2'$ s.t. $(u_1, u_1', u_2', u_3, \ldots, u_D) \in T'$. By (ii), $u_1'$ has weight $D$ and $u_1$ has weight 0 in the weight tuple for $(u_1, u_1', u_2, u_3, \ldots, u_D)$. By (i), it follows that either $u_1'$ or $u_2'$ has  weight $D$ in $(u_1, u_1', u_2', u_3, \ldots, u_D)$.
By the contrapositive of (ii),
  \[p_1(u_1, u_2', u_3, \ldots, u_D) \leq p_1(u_1', u_2', u_3, \ldots, u_D) \]
  Note that for $\Lambda^{\times 2}$-a.a. $(u_1, u_1')$ it is true that for $\Lambda^{\times (D-1)}(U)>0$ of $(u_2, \ldots, u_D) \in U$ and $\Lambda$-a.a. $u_2' \in \bR$ we have $(u_1, u_1', u_2', u_3, \ldots, u_D) \in T'$ and the inequality above holds.
  
  We repeat this argument, "replacing" $u_2, u_3, \ldots, u_D$ with $\Lambda$-almost arbitrary $u_2', u_3', \ldots, u_D'$ and noting the invariant that one of $u_1',\ldots, u_k'$ has weight of $D$ in the weight tuple of $(u_1, u_1', u_2', \ldots, u_k', u_{k+1}, \ldots, u_D) \in T'$ (which means we may apply (ii) to replace $u_{k+1}$ with $\Lambda$-almost arbitrary $u_{k+1}'$). Thus, for $\Lambda^{\times 2}$-a.a. $(u_1, u_1')$,
  \[p_1(u_1, u_2', u_3', \ldots, u_D') \leq p_1(u_1', u_2', u_3', \ldots, u_D') \]
for $\Lambda^{(D-1)}$-a.a. $(u_2', \ldots, u_D')$.

(iv) This follows immediately from (iii) and \eqref{eq17}.

(v) Suppose $f_{\vec{p}}$ is constant on a set $V$ with $\Lambda(V) > 0$. By (iv), can restrict $V$ to a set $V'$ with $\Lambda(V') = \Lambda(V) > 0$ s.t. for all distinct $u_1, u_2, \ldots, u_{D+1} \in V'$, $(u_1, \ldots, u_{D+1}) \in T'$ and
\[ p_1(u_i, u_2, \ldots, \widehat{u_i}, \ldots, \widehat{u_j}, \ldots, u_{D+1}) = p_1(u_j, u_2, \ldots, \widehat{u_i}, \ldots, \widehat{u_j}, \ldots, u_{D+1}) \ \forall i, j \]
We claim this is impossible.

Consider distinct $u_1, \ldots, u_{D+1} \in V'$. Suppose $u_i$ has weight $D$ and $u_k$ has weight 1 in the corresponding weight tuple. Pick $j \notin \{i, k\}$ (which exists since $D+1 \geq 3$). Then $u_j$ has weight 0 in the corresponding weight tuple. But this implies  the contradiction
\[ 1 = p_1(u_i, u_2, \ldots, \widehat{u_i}, \ldots, \widehat{u_j}, \ldots, u_{D+1}) = p_1(u_j, u_2, \ldots, \widehat{u_i}, \ldots, \widehat{u_j}, \ldots, u_{D+1}) = 0 \]
Therefore $f_{\vec{p}}$ is not constant on any set of positive $\Lambda$-measure.
 
 (vi) It suffices to show that for $\Lambda$-a.a. $u_1$, we have
 \[p_1(u_1,u_2,\ldots, u_D) = 1 \]
 for all but a $\Lambda^{\times (D-1)}$-null subset of tuples  $(u_2,\ldots, u_D) \in \bigg(f_{\vec{p}}^{-1}([0, f_{\vec{p}}(u_1)]) \bigg)^{D-1}$.
 
 Suppose not, say there is a set $G$ of $u_1$ of $\Lambda$-positive measure on which 
  \[p_1(u_1,u_2,\ldots, u_D) = 0 \]
for a tuples $(u_2,\ldots, u_D)$ in a set $H_{u_1} \subseteq \bigg(f_{\vec{p}}^{-1}([0, f_{\vec{p}}(u_1)])\bigg)^{D-1}$ of $\Lambda^{\times (D-1)}$-positive measure. By (iv), for $\Lambda$-a.a. $u_1 \in G$, 
\[0= p_1(u_1,u_2,\ldots, u_D) \geq p_1(u_{D+1}, u_2,\ldots, u_D)\]
for $\Lambda$-a.a. $u_{D+1} \in f_{\vec{p}}^{-1}([0, f_{\vec{p}}(u_1)])$ and $\Lambda^{\times (D-1)}$-a.a. $(u_2,\ldots, u_D) \in H_{u_1}$. In particular, there exists $2 \leq i \leq D$ and a subset $H_{u_1}' \subset H_{u_1}$ of positive $\Lambda^D$-measure s.t. the choice for $(u_1, u_2, \ldots, u_D)$ is $u_i$ and s.t. for $\Lambda$-a.a. $u_{D+1} \in f_{\vec{p}}^{-1}([0, f_{\vec{p}}(u_1)])$ we have  
 \begin{equation}\label{eqn9}
 p_1(u_1,u_2,\ldots, u_D) = p_1(u_{D+1}, u_2,\ldots, u_D) = 0
 \end{equation}
 For each such tuple consider the weight of $u_i$ in $(u_1,\ldots, u_{D+1})$; it must be 1 or $D$ and if it were 1 then the choice $(u_1,u_2,\ldots, u_D) \mapsto u_i$ implies none of $u_2,\ldots, u_D$ can have weight $D$ so $u_{D+1}$ must have weight $D$, contradicting $p_1(u_{D+1}, u_2,\ldots, u_D)= 0$. Thus $u_i$ has weight $D$ in $(u_1,\ldots, u_{D+1})$ so we must have the choice $(u_{D+1}, u_2,\ldots, u_D) \mapsto u_i$. Therefore 
 \begin{equation}\label{eqn10}
 p_1(u_i, u_2, \ldots, u_{i-1}, u_{D+1}, u_{i+1}, \ldots, u_D) = 1 \end{equation}
 
 Since \eqref{eqn9} holds for $\Lambda$-a.a. $u_{D+1} \in f_{\vec{p}}^{-1}([0, f_{\vec{p}}(u_1)])$ and $\Lambda^{\times(D-1)}$-a.a. $(u_2,\ldots, u_D) \in H_{u_1}'$, then  there is a set of tuples $(u_2,\ldots, u_{D+1})$ of positive $\Lambda^{\times D}$-measure for which \eqref{eqn9} (hence \eqref{eqn10}) holds for both $(u_2, \ldots, u_{D+1})$ and $(u_2,\ldots, u_{i-1}, u_{D+1}, u_{i+1}, \ldots, u_D, u_i)$ (i.e. it holds with the values $u_i, u_{D+1}$ swapped). Combining the two sets of \eqref{eqn9}, \eqref{eqn10} gives the contradiction
\[0 = p_1(u_i, u_2, \ldots, u_{i-1}, u_{D+1}, u_{i+1}, \ldots, u_D) = 1  \]

(vii) By (v) and (vi), for $\Lambda^{\times D}$-a.a. $(u_1,\ldots,u_D)$, $f_{\vec{p}}(u_i)$ are distinct and 
\[ p_1(u_k, u_1,\ldots, \widehat{u_k}, \ldots, u_D) \geq \mathbf{1}_{\{f_{\vec{p}}(u_k) \geq f_{\vec{p}}(u_i) \ \forall i\}} \ \forall k\]
Consider any such $(u_1,\ldots, u_D)$ and let $u_k$ maximize $f_{\vec{p}}$. Then 
\[ p_1(u_k, u_1,\ldots, \widehat{u_k}, \ldots, u_D) = \mathbf{1}_{\{f_{\vec{p}}(u_k) \geq f_{\vec{p}}(u_i) \ \forall i\}} = 1\]
so the choice for $(u_1,\ldots, u_D)$ is $u_k$. It follows that for all other $j \neq k$,
\[ p_1(u_j, u_1,\ldots, \widehat{u_j}, \ldots, u_D) = 0 = \mathbf{1}_{\{f_{\vec{p}}(u_j) \geq f_{\vec{p}}(u_i) \ \forall i\}}\]
We have shown that the a.s. inequality in (vi) is an a.s. equality.
 \end{proof}
 
 The next lemma shows that (ii) implies (iii) in Theorem \ref{thm4}.

\begin{lemma}\label{Lemma10}
Suppose $\vec{p}$ is a single-edge strategy s.t.
\begin{enumerate}[label=(\roman*)]
\item $f_{\vec{p}}$ is not constant on any set of positive $\Lambda$ measure, and
\item $p_k(u_1,\ldots, u_D) = \textbf{1}_{\{f_{\vec{p}}(u_k) \geq f_{\vec{p}}(u_i) \ \forall i\}} \ \mbox{for $\Lambda^{\times D}$-a.a.} \ (u_1,\ldots, u_D) \ \forall 1 \leq k \leq D$
\end{enumerate}
Then  $\sup_{\phi} E[\sigma_{\vec{p}^{\phi}}] = E[\sigma_{MAX}]$ where the supremum is taken over $\Lambda$-measure-preserving bijections $\phi: [0,1] \rightarrow [0,1]$.
\end{lemma}

\begin{proof}
We begin by constructing a sequence $(\phi_n)$ of $\Lambda$-measure-preserving bijections for which
\begin{equation}\label{24}
     \mathbf{1}_{\{f_{\vec{p}}(\phi_n(u_1)) \geq f_{\vec{p}}(\phi_n(u_i)) \ \forall i\}} \rightarrow \mathbf{1}_{\{u_1 \geq u_i \ \forall i\}}  \ \mbox{$\Lambda^{\times D}$-a.s.}
\end{equation}

Consider any $n \in \bN$. By (i), inverse images of singletons $f_{\vec{p}}(\{c\})$ have $\Lambda$-measure 0. There exist $ 0 = a_0^n \leq a_1^n \leq \cdots \leq a_{D 2^n}^n = 1$ s.t.
\begin{equation}\label{eq24}
    \Lambda ([a_i^n, a_{i+1}^n)) = \Lambda \bigg(f_{\vec{p}}^{-1} \bigg( \bigg[ \frac{i}{2^n}, \frac{i+1}{2^n} \bigg) \bigg)\bigg) \ \forall 0 \leq i \leq D2^n-1
\end{equation} 
The indices stop at $D2^n-1$ since the range of $f_{\vec{p}}$ is a subset of $[0,D]$. Clearly, we can make it so that the partition for $n+1$ refines the one for $n$, for all $n$.

We construct $\phi_n$ by pasting together $\Lambda$-measure-preserving bijections we get from Theorem \ref{thmNishiura} between sets
\[ [a_i^n, a_{i+1}^n) \rightarrow f_{\vec{p}}^{-1} \bigg( \bigg[ \frac{i}{2^n}, \frac{i+1}{2^n} \bigg) \bigg)\]
for $0 \leq i \leq D2^n-1$ for which $f_{\vec{p}}^{-1} ( [ \frac{i}{2^n}, \frac{i+1}{2^n} ) )$ has positive $\Lambda$-measure.  Then for each such $i$, $f_{\vec{p}^{\phi_n}} = f_{\vec{p}} \circ \phi_n$ is a map
\[ [a_i^n, a_{i+1}^n) \rightarrow  \bigg[ \frac{i}{2^n}, \frac{i+1}{2^n} \bigg) \]

Let us show \eqref{24}. It suffices to show that for $\Lambda^{\times 2}$-a.a. $(u_1, u_2)$ we have
\[\mathbf{1}_{\{f_{\vec{p}}(\phi_n(u_1)) \geq f_{\vec{p}}(\phi_n(u_2)) \}} \rightarrow \mathbf{1}_{\{u_1 \geq u_2 \}}\]
since multiplying $D-1$ of these sequences of indicators gives $\eqref{24}$.

By (i), the sets $f_{\vec{p}}^{-1}(\{c\})$ have $\Lambda$-measure 0 hence it follows that for $\Lambda^{\times 2}$-a.a. $(u_1, u_2)$ we have $f_{\vec{p}}(u_1) \neq f_{\vec{p}}(u_2)$. Consider any such $u_1, u_2$. Since $\Lambda([a_i^n, a_{i+1}^n)) \downarrow 0$ as $n \rightarrow \infty$ (because $\Lambda \ll $ Lebesgue measure), then for large enough $n$, $u_1$ and $u_2$ are in different intervals of the form $[a_i^n, a_{i+1}^n)$ and they stay in these intervals (because the partitions get progressively more  refined). Thus $f_{\vec{p}}\circ \phi_n (u_1), f_{\vec{p}}\circ \phi_n (u_2)$ are in different intervals of the form $[\frac{i}{2^n}, \frac{i+1}{2^n})$ and the order is preserved:
\[u_1 < u_2 \Leftrightarrow u_1 \in [a_{i_1}^n, a_{i_1+1}^n), u_2 \in [a_{i_2}^n, a_{i_2+1}^n) \ \mbox{with} \ i_1 < i_2 \Leftrightarrow f_{\vec{p}}\circ \phi_n (u_1) < f_{\vec{p}}\circ \phi_n (u_2)\]
 Therefore we get \eqref{24}.

Recall that $\vec{p}^{\phi_n}$ are themselves consistent single-edge strategies. Using (ii) and \eqref{cdf}, it follows that for any $t \in [0,1]$,
\begin{equation}\label{limit}
\begin{split}
 1-F_{\vec{p}^{\phi_n}}(t) 
    &= D\int_{[t,1] \times [0,1]^{D-1}} p_1(\phi_n(\vec{u})) d\Lambda^{D}(u_1, \ldots, u_{D})  \\
    & =  D\int_{[t,1] \times [0,1]^{D-1}} \mathbf{1}_{\{f_{\vec{p}}(\phi_n(u_1)) \geq f_{\vec{p}}(\phi_n(u_i)) \ \forall i\}} d\Lambda^{D}(u_1, \ldots, u_{D}) 
\end{split}
\end{equation}
By \eqref{24} and  the Bounded Convergence Theorem, this last expression converges to the integral of the density $f_{MAX}$:
\begin{equation*}
\begin{split}
 & D\int_{[t,1] \times [0,1]^{D-1}} \mathbf{1}_{\{f_{\vec{p}}(\phi_n(u_1)) \geq f_{\vec{p}}(\phi_n(u_i)) \ \forall i\}} d\Lambda^{D}(u_1, \ldots, u_{D}) \\
 & \rightarrow D\int_{[t,1] \times [0,1]^{D-1}} \mathbf{1}_{\{u_k \geq u_i \ \forall i\}}d\Lambda^{D}(u_1, \ldots, u_{D}) \\
 & =  \int_{t}^{1} f_{MAX}(u_1) du_1 \\
& =  1- F_{MAX}(t)
\end{split}
\end{equation*}

But we already established that $F_{\vec{q}} \geq F_{MAX}$ for any single-edge strategy $\vec{q}$. Therefore  $F_{\vec{p}^{\phi_n}} \rightarrow F_{MAX}$ pointwise from below hence $E[\sigma_{\vec{p}^{\phi_n}}] \rightarrow E[\sigma_{MAX}]$ by the Bounded Convergence Theorem combined with the tail integral formula for expectation.
\end{proof}

\subsection{Theorem \ref{thm4}}
We wish to prove the rest of Theorem \ref{thm4}, but first we make some useful observations about $\sigma_{MAX}$. 

\begin{lemma}\label{lemma9}
Let $f_{MAX}^-$ be the quantile function
\[ f_{MAX}^-(t) = \inf \{x: t \leq f_{MAX}(x) \} \]
Then treating $f_{MAX}$ itself as a $[0,D]$-valued random variable on the probability space \\ $([0,1], \mathcal{B}([0,1]), \Lambda)$,
\[ P(f_{MAX} \leq t) = f_{MAX}^-(t) \ \forall t\]
In particular,
\[ P(f_{MAX} \leq f_{MAX}(y)) = y \ \forall y \]
\end{lemma}


\begin{proof}
This is trivial since the density $f_{MAX}(y) = Dy^{D-1}$ is invertible so the quantile function  $f_{MAX}^-$ is just its inverse.




\end{proof}

To prove Theorem \ref{thm4}, we need one more short lemma exploring what \\ $\sup_{\phi} E[\sigma_{\vec{p}^{\phi}}] = E[\sigma_{MAX}]$ tells us.

\begin{lemma}\label{lemma10}
Let $\vec{p}$ be a single-edge strategy and  $\phi_n: [0,1] \rightarrow [0,1]$ be $\Lambda$-measure-preserving bijections s.t. $E[\sigma_{\vec{p}^{\phi}}] \uparrow E[\sigma_{MAX}]$. Then there is a subsequence $\phi_{n_j}$ s.t. 
$F_{\vec{p}^{\phi_{n_j}}} (y) \rightarrow F_{MAX}(y)$, $f_{\vec{p}^{\phi_{n_j}}}(y) = f_{\vec{p}} \circ \phi_{n_j}(y) \rightarrow f_{MAX}(y)$ $\Lambda$-a.e. and almost uniformly on compact subsets of $[0,1]$.
\end{lemma}

\begin{proof}
 By the tail integral formula for expectation,
\[ \int_{0}^{1} F_{\vec{p}^{\phi_n}}(t)- F_{MAX}(t)  dt = E[\sigma_{MAX}] - E[\sigma_{\vec{p}^{\phi}}] \downarrow 0\]
where the integrands $F_{\vec{p}^{\phi_n}}(t)- F_{MAX}(t)$ are non-negative from Lemma \ref{lemma3} (i), hence this is $L^1$ convergence. It follows that there exists a subsequence $n_j$ s.t. $F_{\vec{p}^{\phi_{n_j}}}(t) \rightarrow F_{MAX}(t)$ a.e. and a.u. on compact subsets of $[0,1]$. The latter convergence gives us that $f_{\vec{p}^{\phi_{n_j}}}(t) \rightarrow f_{MAX}(t)$ a.e. and a.u. on compact subsets of $[0,1]$.  
\end{proof}

We are now ready to prove Theorem \ref{thm4} in its entirety.

 \begin{manualtheorem}{\ref{thm4}}
Let $\vec{p}$ be a single-edge strategy.  The following are equivalent:
\begin{enumerate}[label=(\roman*)]
\item $\sigma_{\vec{p}}$ is an extreme point
\item Any consistent single-edge strategy achieving $\sigma_{\vec{p}}$  must be deterministic
\item $f_{\vec{p}}$ is not constant on sets of positive $\Lambda$-measure and $\sigma_{\vec{p}}$ is given by the following single-edge strategy $\vec{q}$:
\[ q_k(u_1,\ldots, u_D) = \textbf{1}_{\{f_{\vec{p}}(u_k) \geq f_{\vec{p}}(u_i) \ \forall i\}} \ \mbox{for $\Lambda^{\times D}$-a.a.} \ (u_1,\ldots, u_D) \]
In other words, $\sigma_{\vec{p}}$ is achieved by  the deterministic single-edge strategy "choose whichever label yields a higher value when evaluating the density $f_{\vec{p}}$"
\item $\sup_{\phi} E[\sigma_{\vec{p}^{\phi}}] = E[\sigma_{MAX}]$
\item For $U \sim [0,1]$ if $f_{\vec{p}}(U), f_{MAX}(U)$ as $[0,D]$-valued random variables  on the probability space $([0,1], \mathcal{B}([0,1]), \Lambda)$, then they have the same distribution. That is,
\[P(f_{\vec{p}} \leq t) = P(f_{MAX} \leq t) = \bigg(\frac{t}D\bigg)^{\frac1{D-1}} \ \forall t \in [0,D]\]
\end{enumerate}
\end{manualtheorem}

\begin{proof} 
(i) $\Leftrightarrow $ (ii), (i) $\Rightarrow$ (iii), (iii) $\Rightarrow $ (iv) These are given by Lemmas \ref{lemma6}, \ref{lemma8}, \ref{Lemma10}.

(iv) $\Rightarrow$ (i) Suppose (iv) holds but $\sigma_{\vec{p}}$ is not an extreme point. Since Radon-Nikodym derivatives are additive, then
\[ f_{\vec{p}} = \alpha f_{\vec{q}} + (1-\alpha) f_{\vec{r}}\]
for $\alpha \in (0,1)$ and single-edge strategies $\vec{q}, \vec{r}$ s.t. $f_{\vec{p}} \neq  f_{\vec{q}}$ on a set of positive $\Lambda$ measure.

Take $\phi_n$ $\Lambda$-measure-preserving bijections s.t. $E[\sigma_{\vec{p}^{\phi_n}}] \uparrow E[\sigma_{MAX}]$. Then
\[\alpha E[\sigma_{\vec{q}^{\phi_n}}] + (1-\alpha) E[\sigma_{\vec{r}^{\phi_n}}] \uparrow \alpha E[\sigma_{MAX}] + (1-\alpha) E[\sigma_{MAX}] \]
where $E[\sigma_{\vec{q}^{\phi_n}}], E[\sigma_{\vec{r}^{\phi_n}}] \leq E[\sigma_{MAX}]$. It follows that 
\[E[\sigma_{\vec{q}^{\phi_n}}], E[\sigma_{\vec{r}^{\phi_n}}] \uparrow E[\sigma_{MAX}]\]
By Lemma \ref{lemma10}, there is a subsequence $\phi_{n_j}$ s.t. 
\[ f_{\vec{q}} \circ \phi_{n_j},  f_{\vec{r}} \circ \phi_{n_j} \rightarrow f_{MAX} \]
$\Lambda$-a.e. and $\Lambda$-a.u. on compact subsets of $[0,1]$.

Since $f_{\vec{q}} \neq f_{\vec{r}}$ on a set of positive $\Lambda$-measure then $\exists \epsilon > 0$ s.t. 
\[Q := \{x: |f_{\vec{q}}(x) - f_{\vec{r}}(x)| \geq \epsilon \} \]
has measure $\delta := \Lambda(Q) >0$. Now, by the choice of $\phi_{n_j}$, there is a subset $R \subset [0, 1]$ of $\Lambda$ measure $\leq \frac{\delta}2$ s.t. $f_{\vec{q}}\circ \phi_{n_j}, f_{\vec{r}} \circ \phi_{n_j} \rightarrow f_{MAX}$ uniformly on $[0, 1] \setminus R$. In particular, $\exists J \in \bN$ s.t. for all $j \geq J$ and $y \in [0, 1] \cap R$ we have
\[|f_{\vec{q}} \circ \phi_{n_j}(y) - f_{MAX}(y)|, |f_{\vec{r}} \circ \phi_{n_j}(y) - f_{MAX}(y)| < \frac{\epsilon}2 \Rightarrow |f_{\vec{q}} \circ \phi_{n_j}(y) - f_{\vec{r}} \circ \phi_{n_j}(y) | < \epsilon\]
Thus 
\begin{equation}\label{eq25}
    \phi_{n_J}([0,1] \setminus R) \subseteq [0,1] \setminus Q
\end{equation}
But $\phi_{n_J}$ is $\Lambda$-measure-preserving and 
\begin{align*}
    \Lambda(\phi_{n_J}([0,1] \setminus R)) &= \Lambda([0,1]) - \Lambda(R)\\
    &\geq \Lambda([0,1]-\frac{\delta}2 \\
    &> \Lambda([0,1])-\Lambda(Q) \\
    &= \Lambda([0,1] \setminus Q)
\end{align*}
Contradiction of \eqref{eq25}. Therefore $\sigma_{\vec{p}}$ is an extreme point.

(iv) $\Rightarrow$ (v) Take $\phi_n$ $\Lambda$-measure-preserving bijections s.t. $E[\sigma_{\vec{p}^{\phi_n}}] \uparrow E[\sigma_{MAX}]$. By Lemma \ref{lemma10} there is a subsequence $\phi_{n_j}$ s.t. $f_{\vec{p}^{\phi_{n_j}}} = f_{\vec{p}} \circ \phi_{n_j} \rightarrow f_{MAX}$ $\Lambda$-a.e and a.u. on compact subsets of $[0,1]$. 

Now for any $z \in [0,1]$, 
\[P(f_{\vec{p}} \leq z) = 
\Lambda(f_{\vec{p}}^{-1}([0,z])) = \Lambda( (f_{\vec{p}^{\phi_{n_j}}})^{-1} ([0,z]))\]
We claim this last expression equals $\Lambda(f_{MAX}^{-1}([0,z])) = P(f_{MAX} \leq z)$.

Suppose $P(f_{\vec{p}} \leq z) \neq P(f_{MAX} \leq z)$ for some $z$, say $P(f_{\vec{p}} \leq z) < \Lambda(f_{MAX}^{-1}([0,z]))$; the $>$ case is very similar. By continuity from below and since $\Lambda(f_{MAX}^{-1}(\{z\})) = 0$, there exists $ \delta > 0$ s.t.
\[\Lambda( f_{\vec{p}} ^{-1} ([0,z])) < \Lambda(f_{MAX}^{-1} ([0, z-\delta]))\]
By a.u. convergence on $[0,1]$, there is a subset $R \subset [0, 1]$ off of which $f_{\vec{p}} \circ \phi_{n_j} \rightarrow f_{MAX}$ uniformly s.t. $\Lambda(R)$ is sufficiently small so that for all $j$,
\begin{equation}\label{eq26}
    \Lambda( (f_{\vec{p}} \circ \phi_{n_j})^{-1} ([0,z])) = \Lambda( f_{\vec{p}} ^{-1} ([0,z])) < \Lambda(f_{MAX}^{-1} ([0, z-\delta]) \cap [0,1] \setminus R)
\end{equation}
Uniform convergence gives a $J$ s.t. for all $j \geq J$ and all $t \in f_{MAX}^{-1} ([0, z-\delta]) \cap [0,1] \setminus R$,
\[ f_{\vec{p}} \circ \phi_{n_j}(t) \in (f_{MAX}(t)- \delta, f_{MAX}(t) + \delta) \cap [0,\infty) \subseteq [0, z] \]
(here we use the fact that $f_{\vec{p}}$ is a non-negative function). Thus
\[f_{MAX}^{-1} ([0, z-\delta]) \cap [0,1] \setminus R \subseteq  (f_{\vec{p}} \circ \phi_{n_j}) ^{-1} ([0,z]) \]
which is a contradiction of \eqref{eq26}. Therefore
\[P(f_{\vec{p}} \leq z) = P(f_{MAX} \leq z) =  f_{MAX}^-(z) \]
where the last inequality holds by Lemma \ref{lemma9}. The expression on the right is precisely $ (\frac{z}D)^{\frac1{D-1}}$.

(v) $\Rightarrow$ (iv) We construct $\Lambda$-measure-preserving bijections 
 $\phi_n : [0,1] \rightarrow [0,1]$ s.t.\\ $E[\sigma_{\vec{p}^{\phi_n}}] \rightarrow E[\sigma_{MAX}]$. 
 
 Consider any $n$. By (v) combined with Lemma \ref{lemma9}, for any $z \in \{0,1,\ldots, 2^n-1\}$ we have
 \begin{align*}
     &\Lambda \bigg( f_{\vec{p}}^{-1} \bigg( \bigg[f_{MAX} \bigg(\frac{z}{2^n} \bigg), f_{MAX} \bigg(\frac{z+1}{2^n} \bigg) \bigg) \bigg) \bigg) \\
     &= P\bigg(f_{MAX} \bigg(\frac{z}{2^n} \bigg) \leq f_{\vec{p}} < f_{MAX} \bigg(\frac{z+1}{2^n} \bigg) \bigg) \\
     &= P\bigg(f_{MAX} \bigg(\frac{z}{2^n} \bigg) \leq f_{MAX} < f_{MAX} \bigg(\frac{z+1}{2^n} \bigg) \bigg) \\
     &= \Lambda\bigg( \bigg[\frac{z}{2^n}, \frac{z+1}{2^n} \bigg) \bigg)
 \end{align*}
 We construct $\phi_n$ by pasting together $\Lambda$-measure-preserving bijections we get from Theorem \ref{thmNishiura} between sets
 \[\bigg[\frac{z}{2^n}, \frac{z+1}{2^n} \bigg) \rightarrow f_{\vec{p}}^{-1} \bigg( \bigg[f_{MAX} \bigg(\frac{z}{2^n} \bigg), f_{MAX} \bigg(\frac{z+1}{2^n} \bigg) \bigg) \bigg) \]
Then for each $z$, $f_{\vec{p}^{\phi_n}} = f_{\vec{p}} \circ \phi_n$ is a map
  \[\bigg[\frac{z}{2^n}, \frac{z+1}{2^n} \bigg) \rightarrow  \bigg[f_{MAX} \bigg(\frac{z}{2^n} \bigg), f_{MAX} \bigg(\frac{z+1}{2^n} \bigg) \bigg)  \]
We now compute the expectation of each $X(\vec{p}^{\phi_n})$:
\begin{align*}
E[X(p^{MAX}) ] &\geq E[X_0(\vec{p}^{\phi_n}) ] \\
&= \sum_{z=0}^{2^n-1}  \int_{\frac{z}{2^n}}^{\frac{z+1}{2^n}} f_{\vec{p}}(\phi_n(t)) t\ dt\\
&\geq \sum_{z=0}^{2^n-1} \int_{\frac{z}{2^n}}^{\frac{z+1}{2^n}} f_{MAX} \bigg(\frac{z}{2^n} \bigg) t \ dt
\end{align*}
But $f_{MAX}(t) t$ is nondecreasing hence is of bounded variation on $[0,1]$ hence this lower bound converges to $E[X(p^{MAX}) ]$ as  $n \rightarrow \infty$. Thus
\[ E[\sigma_{MAX}] = E[X(p^{MAX}) ] = \sup_n E[X(\vec{p}^{\phi_n}) ] =\sup_{\phi} E[\sigma_{\vec{p}^{\phi}} ]  \]
 as desired.
\end{proof}

By the Krein-Milman Theorem, $\mathcal{R}$ is exactly the closed convex hull of its extreme points. Thus we obtain a description of the set of limit points of empirical measures.

As a final remark, observe that we only  used the compactness of the space $\mathcal{M}_1$ of probability measures on [0,1] to reduce from generic strategies to single-edge strategies. Furthermore, the fact that $\Lambda$ has nice formulas for its cdf and pdf was convenient but unnecessary. The proof of this theorem could be tweaked to hold with $\Lambda$ replaced by an arbitrary distribution $\theta$ on $\bR$ with finite mean. We could thus get a similar characterization for the extreme points of $\{\sigma_{\psi}: \mbox{single-edge strategies} \ \psi\}$. The only caveats would be that this set might not coincide with $\{\mbox{limit points of } \ \frac1n \mu_{0 \rightarrow n}(\chi): \ \mbox{strategies} \ \chi\}$ and that the value distribution of the densities $f_{\vec{p}}$ of extreme points might not have as nice a form as $D \cdot Beta(1,D)$.

\subsection{The Discrete Case with Example}\label{discrete}
The same argument with the weight tuples can be used to show a discrete version of Theorem \ref{thm4}, where the i.i.d. labels $U^j$ are $Unif\{1,\ldots,K\}$. 

\begin{manualtheorem}{13'}\label{13discrete}
Let $\vec{p}$ be a single-edge strategy.  The following are equivalent
\begin{enumerate}[label=(\roman*)]
\item $\sigma_{\vec{p}}$ is an extreme point
\item Any consistent single-edge strategy achieving $\sigma_{\vec{p}}$  must be deterministic.
\item $\sigma_{\vec{p}}$  is given by the single-edge strategy "choose whichever label is maximal with respect to the ordering $\alpha(1) < \cdots < \alpha(K)$"  for some $\alpha \in S_K$.
\item There exists a permutation $\beta \in S_K$ s.t. $\vec{p}(\beta(\vec{u})) = \vec{p}^{MAX}(\vec{u}) \ \forall \vec{u}$.
\item $\sigma_{\vec{p}}$ has a probability mass function whose value distribution is 
\[\bigg\{\frac1{K^D}, \frac{2^D-1}{K^D}, \frac{3^D-2^D}{K^D}, \ldots, \frac{K^D-(K-1)^D}{K^D}\bigg\}\]
\end{enumerate}
\end{manualtheorem}

Let us begin by highlighting the main differences between Theorems \ref{thm4} and \ref{13discrete}. In the continuous case,  the extreme points are those $\sigma_{\vec{p}}$ achieved by the deterministic single-edge strategy of choosing whichever $u_i$ maximizes the value of the density $f_{\vec{p}}$, or equivalently those $\sigma_{\vec{p}}$ whose density has the same value distribution as $\sigma_{MAX} = Dx^{D-1}dx$ (namely $D \cdot \mbox{Beta}(1,D)$). Contrast this with the discrete case, where the  extreme points are those $\sigma_{\vec{p}}$ achieved by the deterministic single-edge strategy of choosing whichever $u_i$ maximizes the value of the probability mass function, or equivalently those $\sigma_{\vec{p}}$ whose pmf has the same value distribution as $\sigma_{MAX} = \frac1{K^D} \delta_1 + \frac{2^D-1}{K^D}\delta_2+ \cdots + \frac{K^D-(K-1)^D}{K^D} \delta_K$. Furthermore, in the discrete case, if we write the extremal $\sigma_{\vec{p}}$ as
\[\sigma_{\vec{p}} = \frac1{K^D} \delta_{\alpha(1)} +\frac{2^D-1}{K^D}\delta_{\alpha(2)} + \cdots + \frac{K^D-(K-1)^D}{K^D} \delta_{\alpha(K)}  \]
for some permutation $\alpha \in S_K$ then clearly the single-edge strategy of choosing whichever $u_i$ maximizes the value of the pmf can be equivalently described as the single-edge strategy of choosing whichever $u_i$ is maximal according to the ordering $\alpha(1) < \alpha(2) < \cdots < \alpha(K)$. It is easy to see that this $\alpha$ satisfies $\vec{p}(\alpha^{-1}(\vec{u})) = \vec{p}^{MAX}(\vec{u})$ (which is where (iv) above comes from).

In the discrete setting, the existence of this $\alpha \in S_K$ leads  to a natural bijection between the extreme points of the permutohedron and of $\mathcal{R}$, the set of achievable distributions $\{\sigma_{\vec{p}}\}$, by mapping $\sigma_{\vec{p}}$  to the ordering $\alpha(1) < \cdots < \alpha(K)$ associated with it.

Since convex combinations of single-edge strategies translate to convex combinations of the distributions $\sigma_{\vec{p}}$ then this bijection extends to a bijection between the permutohedron and $\mathcal{R}$. Of course, the permutohedron is the same for different $D$ but the bijection depends on $D$.

It is also worth noting that our proof of Theorem \ref{thm4} adapted to this discrete setting becomes much simpler and more intuitive. Recall we showed that a consistent and deterministic single-edge strategy is one which induces an ordering on the possible values of the edge labels in the sense that $y$ "dominates" $x$ if and only if the choice for $\vec{u}$ is never $x$ if both $x$ and $y$ appear in $\vec{u}$. In the discrete case, there are only $K$ possible values of the edge labels so immediately this gives the ordering $\alpha(1) < \cdots < \alpha(K)$ from which Theorem \ref{13discrete} follows.

As an example, let us briefly work through the $D=2, K =4$ discrete case.
\begin{center}
\captionsetup{type=figure}
\includegraphics[scale=0.4]{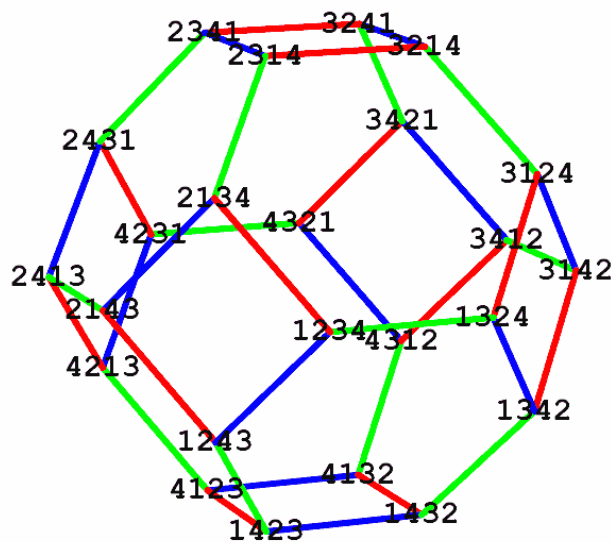}
\captionof{figure}{Permutohedron of order 4, \cite{holroyd}}
\end{center}
The bijection between the extreme points of the permutohedron of order 4 and the extreme points of the set of consistent single-edge strategies goes as follows: 
\[ 1342 \mapsto \sigma_{1342} = \frac1{16} \delta_1 + \frac{2^2-1}{16} \delta_3 + \frac{3^2-2^2}{16} \delta_4 + \frac{4^2-3^2}{16} \delta_2\]
where $\sigma_{1342}$ is the distribution obtained by  choosing whichever of the samples $u^1, u^2 \in \{1,2,3,4\}$ is maximal with respect to the ordering $1 < 3 < 4 < 2$.

It is plain to see that all the deterministic consistent single-edge strategies correspond to such an ordering, and the corresponding distribution has a pmf with the same value distribution $\{\frac1{16},\frac3{16},\frac5{16},\frac7{16}\}$.

\section{Grid Entropy in this Model}\label{grid}
In this short section, we compute the grid entropy of the extreme points described in the previous section and we describe grid entropy in general using a simplified formula for the Gibbs Free Energy.

\subsection{Grid Entropy of Extreme Points}
Recall that the extreme points of $\{\sigma_{\vec{p}}\}$ are given by the single-edge strategy of choosing whichever edge label maximizes the density $f_{\vec{p}}$. We show that such extreme points $\sigma_{\vec{p}}$ have grid entropy 0. 

However, we first need a short lemma about the partial averages of the expectations of order statistics being bounded away from 0.

\begin{lemma}\label{orderStats}
Let $Z_i \sim \theta, i \geq 0$ be i.i.d. random variables s.t. $\theta$ is a distribution on $\bR$ with  cdf $F_{\theta}$ satisfying $F_{\theta}(0) = 0$. Consider the order statistics $Z_{0:n} \leq \cdots \leq Z_{n:n}$. Then there exists a constant $C> 0$ s.t. for large enough $n$,
\[ \frac1n \sum_{k=0}^{\lfloor n \epsilon \rfloor} E[Z_{k:n}] \geq C \]
\end{lemma}

\begin{remark} 
We allow for the case where some of these expected values may be $\infty$.
\end{remark}

\begin{proof}
Observe that half the terms in this sum are bounded below by $E[Z_{\lfloor \frac{n \epsilon}2 \rfloor:n}]$ hence
\[ \frac1n \sum_{k=0}^{\lfloor n \epsilon \rfloor} E[Z_{k:n}] \geq \frac1n \bigg \lfloor \frac{n \epsilon}2 \bigg \rfloor E[Z_{\lfloor \frac{n \epsilon}2 \rfloor:n}]  \]
Thus it suffices to find $C>0$ s.t.  for large $n$, $E[Z_{\lfloor \frac{n \epsilon}2 \rfloor:n}] \geq \frac12 C$.

 By right-continuity of the cdf $F_{\theta}$, we can take $C > 0$ s.t. $F_{\theta}(C) \leq \frac18 \epsilon$.
 
 Consider the i.i.d. Bernoulli random variables $ \textbf{1}_{\{Z_i \leq C\}}$ with success probability $F_{\theta}(C)$. Then Markov's Inequality yields 
\[P(Z_{\lfloor \frac{n \epsilon}2 \rfloor:n} \leq C) = P \bigg(\#\mbox{successes in Bin($n, F_{\theta}(C)$)} \geq \bigg\lfloor \frac{n \epsilon}2 \bigg \rfloor \bigg) \leq \frac{n F_{\theta}(C)}{\lfloor \frac{n \epsilon}2  \rfloor} \leq \frac12 \]
for large $n$ by the choice of $C$. It follows that
\[ E[Z_{\lfloor \frac{n \epsilon}2 \rfloor:n}] \geq \frac12 C  \]
for large $n$, which completes the proof.
\end{proof}

\begin{theorem}\label{entropy0}
Fix $\tau: [0,1] \rightarrow \bR$ measurable and bounded s.t. $\tau$ is not constant on sets of positive $\Lambda$ measure, and consider the single-edge strategy of picking the maximal $\tau(U^j)$ over $1 \leq j \leq D$, given by
\[ \vec{p}(u_1,\ldots,u_D) = \textbf{1}_{\{\tau(u_k) \geq \tau(u_i) \ \forall i\}} \]
 Then $||\sigma_{\vec{p}}|| = 0$.
\end{theorem} 

\begin{remark}
The types of strategies considered in this theorem are deterministic so the resulting $\sigma_{\vec{p}}$ are all extreme points by Theorem  \ref{thm4}. On the other hand, Theorem \ref{thm4} establishes that all extreme points can be realized as $\sigma_{\vec{p}}$ for a single-edge strategy $\vec{p}$ choosing whichever observed label maximizes the density $f_{\vec{p}} \in [0,D]$ where $f_{\vec{p}}(\mbox{Unif}[0,1])$ is the $D \cdot \mbox{Beta}(D,1)$ distribution. Thus  Theorem \ref{entropy0} captures what happens for all extreme points.
\end{remark}

\begin{proof}
Suppose $||\sigma_{\vec{p}}|| > \delta > 0$. Let $\alpha = \frac{\delta}{\ln 2}$. For $n \in \bN$ and $ 1 \leq m_n \leq \lfloor e^{n \delta} \rfloor = \lfloor 2^{n \alpha} \rfloor$ consider the event-dependent  paths $\pi_{n,m_n}$ corresponding to 
\[ \min_{\pi: 0\rightarrow n}^{m_n} \rho\bigg(\frac1n \mu_{\pi}, \sigma_{\vec{p}} \bigg)  \]

Since $\delta < ||\sigma_{\vec{p}}||$ then by definition of grid entropy,
\[ \min_{\pi: 0\rightarrow n}^{\lfloor e^{n\delta} \rfloor } \rho\bigg(\frac1n \mu_{\pi}, \sigma_{\vec{p}} \bigg)  \rightarrow 0 \ \mbox{a.s.}\]
hence $\frac1n \mu_{\pi_{n, m_n}} \Rightarrow \sigma_{\vec{p}}$ a.s. regardless of the sequence $(m_n)$.

Also for $i \geq 0$ define the random variables
\[Y_i := \max \limits_{1 \leq j \leq D} (\tau(U_i^j)), Y_i' = \mbox{2nd} \ \max \limits_{1 \leq j \leq D} (\tau(U_i^j)), Z_i := Y_i - Y_i'\]

Now it is  a classic result that for $\epsilon > 0$,
\[ \sum_{i=0}^{\lfloor n\epsilon \rfloor} \binom{n}{i} \leq 2^{nL} \ \mbox{with} \ L = L(\epsilon)= \epsilon \log \epsilon + (1-\epsilon) \log (1-\epsilon) \]
Take $\epsilon > 0$ small enough so that $L +\epsilon \log(D-1) < \frac12 \alpha$ and take $N \in \bN$ so that $ \forall n \geq N$,  $2^{nL} (D-1)^{n\epsilon} < \lfloor 2^{n \alpha}\rfloor$.

Consider any $n \geq N$. Then the number of paths $\pi: 0 \rightarrow n$ with $< \lceil n\epsilon \rceil$ of its edges   not having the maximal edge label in their trial  is at most
\[ \sum_{i=0}^{\lfloor n\epsilon \rfloor} \binom{n}{i} (D-1)^{i} \leq 2^{nL} (D-1)^{n\epsilon} < \lfloor 2^{n \alpha}  \rfloor \]
By the Pigeonhole Principle, there is an event-dependent path $\pi_{n, m_n}: \vec{0} \rightarrow n$ s.t.  at least $\lceil n\epsilon \rceil$ of its edges do not have the maximal edge label in their trial. Let $\pi_{n, m_n}$ have edges $e_0^{j_0}, \ldots, e_{n-1}^{j_{n-1}}$ and let $I_n \subseteq \{0,\ldots,n-1\}, |I_n| = \lceil n\epsilon \rceil$ be a set of indices $i$ for which $\tau(U_i^{j_i}) \neq Y_i$.

We  compute an upper bound for the passage time along $\pi_{n, m_n}$ (with respect to $\tau$) by splitting the sum over the edges with index in $I_n$ and those not in $I_n$: 
\begin{align*}
\langle \tau, \frac1n \mu_{\pi_{n,m_n}} \rangle &= \frac1n\sum_{i=0}^{n-1} \tau(U_i^{j_i}) \\
& \leq \frac1n \sum_{i\in I_n^C} \max_{1\leq j \leq D} (\tau(U_i^j)) + \frac1n \sum_{i \in I_n} \mbox{2nd} \max_{1\leq j \leq D} (\tau(U_i^j)) \\
&=  \frac1n \sum_{i\in I_n^C} Y_i + \frac1n \sum_{i \in I_n} Y_i'
\end{align*}
On the other hand,  the passage time along the $\tau$-optimal path $\pi_{0\rightarrow n}(\vec{p})$ is
\begin{equation*}
\langle \tau, \frac1n \mu_{\pi_n(\vec{p})} \rangle = \frac1n\sum_{i=0}^{n-1} \max_{1\leq j \leq D} (\tau(U_i^j))  = \frac1n\sum_{i=0}^{n-1} Y_i
\end{equation*}
hence
\begin{equation}\label{lowerB}
\langle \tau, \frac1n \mu_{\pi_n(\vec{p})} \rangle - \langle \tau, \frac1n \mu_{\pi_{n,m_n}} \rangle \geq \frac1n \sum_{i \in I_n} (Y_i - Y_i') = \frac1n \sum_{i \in I_n} Z_i \geq \frac1n \sum_{k=0}^{\lfloor n\epsilon \rfloor} Z_{k:n} 
\end{equation}
Now $Z_i$ are i.i.d., non-negative and satisfy
\begin{align*}
 P(Z_i = 0) &= P\bigg(\max \limits_{1 \leq j \leq D} (\tau(U_i^j)) = \mbox{2nd} \ \max \limits_{1 \leq j \leq D} (\tau(U_i^j))\bigg) \\
 &\leq P(\exists 1 \leq j_1 < j_2 \leq D \ \mbox{s.t.} \ \tau(U_i^{j_1}) = \tau(U_i^{j_2}) )\\
 &=0
\end{align*}
since $\tau$ is not constant on sets of positive $\Lambda$ measure. Thus we can take the expectation in \eqref{lowerB} and apply Lemma \ref{orderStats} to get that $\exists C > 0$ s.t.
\begin{equation}\label{EQ21}
E \bigg[ \langle \tau, \frac1n \mu_{\pi_n(\vec{p})} \rangle -\langle \tau, \frac1n \mu_{\pi_{n,m_n}} \rangle \bigg]  \geq  \frac1n \sum_{k=0}^{\lfloor n\epsilon \rfloor} E[Z_{k:n}] \geq C
\end{equation}
for large $n$.

Now, by assumption, $\frac1n \mu_{\pi_{n,m_n}}  \Rightarrow \sigma_{\vec{p}}$ a.s.. On the other hand, by nature of the model, the $\tau$-optimal length $m$ path $\pi_{m}(\vec{p})$ contains the $\tau$-optimal length $n$ path $\pi_n(\vec{p})$ for any $m \geq n$; thus we can apply the Glivenko-Cantelli Theorem (Theorem \ref{thm1}) to get that the empirical measures $\frac1n \mu_{\pi_{n}(\vec{p})} $ converge weakly to $\sigma_{\vec{p}}$ a.s.. Recall from Section \ref{coupling} that a.s. the pushforward $\tau_{\ast}$ preserves weak limits of empirical measures so
\begin{align*} 
\langle \tau, \frac1n \mu_{\pi_n(\vec{p})} \rangle -\langle \tau, \frac1n \mu_{\pi_{n,m_n}} \rangle &= \langle 1, \tau_{\ast}(\frac1n \mu_{\pi_n(\vec{p})}) \rangle -\langle 1, \tau_{\ast}(\frac1n \mu_{\pi_{n,m_n}}) \rangle \\
&\rightarrow \langle 1, \tau_{\ast}(\sigma_{\vec{p}}) \rangle - \langle 1, \tau_{\ast}(\sigma_{\vec{p}}) \rangle \\
&= 0  
\end{align*}
a.s.. But $\tau$ is bounded so by the Bounded Convergence Theorem we get
\[E \bigg[ \langle \tau, \frac1n \mu_{\pi_n(\vec{p})} \rangle -\langle \tau, \frac1n \mu_{\pi_{n,m_n}} \rangle \bigg] = 0\]
which contradicts \eqref{EQ21}. Thus $||\sigma_{\vec{p}}||=0$. 
\end{proof}

\subsection{Grid Entropy via Gibbs Free Energy}
Fix $\beta > 0$. Suppose $\tau: [0,1] \rightarrow \bR$ is a bounded measurable function. From the definition of $\beta$-Gibbs Free Energy,
\[ G^{\beta}(\tau) := \lim_{n \rightarrow \infty} \frac1n \log \sum_{\pi: 0 \rightarrow n} e^{\beta T(\pi)} = \lim_{n \rightarrow \infty} \frac1n \sum_{i=0}^{n-1} \log \sum_{j=1}^D e^{\beta \tau(U_i^j)} \ \mbox{a.s.} \]
But $\log \sum \limits_{j=1}^D e^{\beta \tau(U_i^j)}$ are i.i.d. in $i$ hence by the SLLN,
\[ G^{\beta} (\tau) = E \bigg[\log \sum_{j=1}^D e^{\beta \tau(U^j)} \bigg] \ \mbox{a.s.}\]
where $U^j$ are i.i.d. Unif[0,1].

Of course, grid entropy is simply the negative convex conjugate of $\beta$-Gibbs Free Energy by Theorem \ref{gridEntropyPart1}:
\begin{equation} -||\nu|| = \sup_{\tau} \bigg[ \beta \langle \tau, \nu \rangle - G^{\beta}(\tau) \bigg] = \sup_{\tau} \bigg[\beta \langle \tau, \nu \rangle - E \bigg[\log \sum_{j=1}^D e^{\beta \tau(U^j)} \bigg] \bigg]  \ \forall \nu \in \mathcal{M}
\end{equation}
where the supremum is over bounded measurable functions $\tau:[0,1] \rightarrow \bR$ and where $\langle \tau, \nu \rangle$ denotes the integral $ \int_0^1 \tau(u) d\nu$.

\section{Next Steps}
We have characterized the extreme points of the set of limit points of empirical measures and have shown that these extreme points have grid entropy 0. A natural next question is whether only the extreme points have grid entropy 0.

Recalling that grid entropy is concave, another key question to ask is whether it is strictly concave in the model used in this paper; this might provide insights on whether it is strictly concave in general, which would imply that the Gibbs Free Energy is strictly concave, a major open research problem. 

 Furthermore, it would be nice to compute grid entropy of non-extremal points, and also try to describe the subset of $\mathcal{R}$ of maximizers for  the variational formula for the Gibbs Free Energy presented in Section 5.2 of \cite{gatea}. 
 
 These are all questions well worth exploring in the future.
\section{Acknowledgments}
I would like to thank my friends and family for their continual support during my graduate studies. Special thanks goes to my advisor B\'{a}lint Vir\'{a}g, for his patience, guidance, and optimism in the face of setbacks. Finally, this work would not have been possible without the funding from my NSERC Canadian Graduate Scholarship-Doctoral.

\bibliography{biblio}

\end{document}